\documentclass[%
  a4paper,
  onecolumn,
  colorlinks,
]{preprint}

\usepackage[english]{babel}

\usepackage{amsmath}
\usepackage{amssymb}
\usepackage{amsthm}

\usepackage[linesnumbered, lined, algoruled, algo2e]{algorithm2e}

\usepackage{graphicx}

\usepackage{tikz}
  \usetikzlibrary{backgrounds}
  \usetikzlibrary{calc}
  \usetikzlibrary{positioning}

\usepackage{pgfplots}
  \pgfplotsset{compat = 1.17, colormap name = viridis}
  \usepgfplotslibrary{fillbetween}
  \usepgfplotslibrary{external}
  \tikzset{external/system call = {%
    pdflatex \tikzexternalcheckshellescape
      -halt-on-error
      -interaction=batchmode
      -jobname "\image" "\texsource"}}
  \tikzexternaldisable

\usepackage{caption}
\usepackage{subcaption}

\newcommand{\trans}{\ensuremath{\mkern-1.5mu\mathsf{T}}}
\newcommand{\ddt}{\ensuremath{\frac{\operatorname{d}}{\operatorname{d}t}}}
\DeclareMathOperator{\mrank}{rank}
\DeclareMathOperator{\Reyn}{Re}
\DeclareMathOperator{\mspan}{span}

\newcommand{\R}{\ensuremath{\mathbb{R}}}
\newcommand{\N}{\ensuremath{\mathbb{N}}}
\newcommand{\E}{\ensuremath{\mathbb{E}}}

\newcommand{\nh}{\ensuremath{N}}
\newcommand{\np}{\ensuremath{p}}
\newcommand{\nT}{\ensuremath{T}}
\newcommand{\nr}{\ensuremath{r}}
\newcommand{\nmin}{\ensuremath{n_{\min}}}

\newcommand{\Ar}{\ensuremath{\widehat{A}}}
\newcommand{\Br}{\ensuremath{\widehat{B}}}
\newcommand{\Kr}{\ensuremath{\widehat{K}}}
\newcommand{\Xr}{\ensuremath{\widehat{X}}}
\newcommand{\Yr}{\ensuremath{\widehat{Y}}}
\newcommand{\Xt}{\ensuremath{\widetilde{X}}}
\newcommand{\Yt}{\ensuremath{\widetilde{Y}}}
\newcommand{\xr}{\ensuremath{\hat{x}}}
\newcommand{\xt}{\ensuremath{\tilde{x}}}

\newcommand{\Xn}{\ensuremath{X^{\operatorname{n}}}}
\newcommand{\Un}{\ensuremath{U^{\operatorname{n}}}}
\newcommand{\Yn}{\ensuremath{Y^{\operatorname{n}}}}
\newcommand{\xn}{\ensuremath{x^{\operatorname{n}}}}

\newcommand{\un}{\ensuremath{u^{\operatorname{n}}}}
\newcommand{\yn}{\ensuremath{y^{\operatorname{n}}}}

\newcommand{\xs}{\ensuremath{\bar{x}}}
\newcommand{\us}{\ensuremath{\bar{u}}}

\newcommand{\Ocal}{\ensuremath{\mathcal{O}}}
\newcommand{\Vcal}{\ensuremath{\mathcal{V}}}
\newcommand{\Xcal}{\ensuremath{\mathcal{X}}}

\newcommand{\Jx}{\ensuremath{J_{\operatorname{x}}}}
\newcommand{\Ju}{\ensuremath{J_{\operatorname{u}}}}

\newcommand{\orth}{\ensuremath{\texttt{orth}}}

\newcommand{\datatrip}{\ensuremath{(U, X, Y)}}
\newcommand{\datatripn}{\ensuremath{(\Un, \Xn, \Yn)}}
\newcommand{\datatripnj}{\ensuremath{(\Un_{j}, \Xn_{j}, \Yn_{j})}}
\newcommand{\datatripnjone}{\ensuremath{%
  (\Un_{j - 1}, \Xn_{j - 1}, \Yn_{j - 1})}}
\newcommand{\datatripr}{\ensuremath{(U, \Xr, \Yr)}}
\newcommand{\datatript}{\ensuremath{(U, \Xt, \Yt)}}

\theoremstyle{plain}\newtheorem{proposition}{Proposition}
\theoremstyle{plain}\newtheorem{lemma}{Lemma}
\theoremstyle{plain}\newtheorem{theorem}{Theorem}
\theoremstyle{plain}\newtheorem{corollary}{Corollary}
\theoremstyle{definition}\newtheorem{remark}{Remark}

\newcommand{\plotfontsize}{\footnotesize}

\definecolor{matlabblue}{HTML}{0072BD}
\definecolor{matlaborange}{HTML}{D95319}
\definecolor{matlabyellow}{HTML}{EDB120}
\definecolor{matlabpurple}{HTML}{7E2F8E}
\definecolor{matlabgreen}{HTML}{77AC30}
\definecolor{matlablightblue}{HTML}{4DBEEE}
\definecolor{matlabred}{HTML}{A2142F}

\tikzstyle{sline} = [
  solid,
  line width = 1.5pt
]

\tikzstyle{samplecol} = [
  line width   = 1.3pt,
  fill opacity = 0.75
]

\pgfplotscreateplotcyclelist{stateslist}{
  {matlabblue, sline},
  {matlaborange, sline},
  {matlabpurple, sline},
  {matlabgreen, sline}
}

\pgfplotscreateplotcyclelist{convergencelist}{
  {matlabblue, sline, dashed},
  {matlaborange, sline, dotted},
  {matlabgreen, sline}
}

\pgfplotscreateplotcyclelist{samplelist}{
  {draw = matlabblue, fill = matlabblue, samplecol},
  {draw = matlaborange, fill = matlaborange, samplecol},
  {draw = matlabpurple, fill = matlabpurple, samplecol},
  {draw = matlabgreen, fill = matlabgreen, samplecol}
}

\begin{document}

%%%%%%%%%%%%%%%%%%%%%%%%%%%%%%%%%%%%%%%%%%%%%%%%%%%%%%%%%%%%%%%%%%%%%%%%%%%%%%%%
% PAPER INFORMATION.                                                           %
%%%%%%%%%%%%%%%%%%%%%%%%%%%%%%%%%%%%%%%%%%%%%%%%%%%%%%%%%%%%%%%%%%%%%%%%%%%%%%%%

\title{An adaptive data sampling strategy for stabilizing dynamical systems via
  controller inference}

\author[$\ast$]{Steffen W. R. Werner}
\affil[$\ast$]{Department of Mathematics, Division of Computational Modeling and
  Data Analytics, and National Security Institute, Virginia Tech,
  Blacksburg, VA 24061, USA.\authorcr
  \email{steffen.werner@vt.edu}, \orcid{0000-0003-1667-4862}
}

\author[$\dagger$]{Benjamin Peherstorfer}
\affil[$\dagger$]{Courant Institute of Mathematical Sciences,
  New York University, New York, NY 10012, USA.\authorcr
  \email{pehersto@cims.nyu.edu}, \orcid{0000-0002-1558-6775}}

\shorttitle{Adaptive data sampling for control}
\shortauthor{S.~W.~R.~Werner, B.~Peherstorfer}
\shortdate{2026-02-10}
\shortinstitute{}



\abstract{%
  Learning stabilizing controllers from data is an important task in
  engineering applications; however, collecting informative data 
  is challenging because unstable systems often lead to rapidly growing or
  erratic trajectories.
  In this work, we propose an adaptive sampling scheme that generates data
  while simultaneously stabilizing the system to avoid instabilities during the
  data collection.
  Under mild assumptions, the approach provably generates data sets that are
  informative for stabilization and have minimal size.
  The numerical experiments demonstrate that controller inference with the
  novel adaptive sampling approach learns controllers with up to one order
  of magnitude fewer data samples than unguided data generation.
  The results show that the proposed approach opens the door to stabilizing
  systems in edge cases and limit states where instabilities often occur and
  data collection is inherently difficult.
}

\novelty{%
  We introduce an adaptive sampling scheme for the
  data-driven construction of stabilizing feedback controllers for nonlinear
  dynamical systems.
  The proposed method generates informative data while simultaneously
  stabilizing the system to avoid instabilities during the data collection,
  which otherwise could render the collected data meaningless.
}

\maketitle

%%%%%%%%%%%%%%%%%%%%%%%%%%%%%%%%%%%%%%%%%%%%%%%%%%%%%%%%%%%%%%%%%%%%%%%%%%%%%%%%
% INTRODUCTION.                                                                %
%%%%%%%%%%%%%%%%%%%%%%%%%%%%%%%%%%%%%%%%%%%%%%%%%%%%%%%%%%%%%%%%%%%%%%%%%%%%%%%%

\section{Introduction}%
\label{sec:intro}

The stabilization of dynamical systems is an essential building block in
many applications ranging from the control of power networks to chemical
reactors to fluid structure interactions.
When explicit model descriptions are unavailable, data has to be used for
the design of suitable stabilizing controllers.
However, the collection of data from unstable systems is a challenging task
since the dynamical instabilities quickly render the collected data
uninformative.
In this work, we consider the design of informative data sets for the task of
system stabilization via an adaptive data sampling scheme.

Many approaches have been considered for the task of data-driven control
ranging from model-free parameter tuning
approaches~\cite{CamLS02, FliJ13, LeqGMetal03, SafT95}
to scientific machine learning and reinforcement learning
methods~\cite{HamXY21, PerUS21, KraTR19, WerP25}.
With the rise of data-driven reduced-order modeling, a two-step procedure
has been established in which first, a model of the underlying dynamics
is learned, before this model is used for the design of controllers via
classical model-based approaches.
Popular data-driven modeling methods for the task of control include
among others the dynamic mode decomposition and operator
inference~\cite{Sch10, TuRLetal14, PehW16, Peh20, SwiKHetal20, BerP24},
sparse identification methods~\cite{BruPK16, SchTW18}, and
the Loewner framework~\cite{MayA07, SchU16, PehGW17, SchUBetal18,
AntGI16, GosA18}.
In recent years, the concept of data informativity has emerged as an effective
tool for solving a variety of data-driven tasks including system
stabilization~\cite{DePT20, VanETetal20, WerP24, WerP23a}.
This framework is well suited for the development of new data-driven approaches
as it is based on the question of how much data is needed to fulfill a
considered task.
While there is such a variety of different approaches for designing controllers
from given data, most of these methods assume that the data contain enough
information about the system to solve the task at hand.
Here, we aim to design such informative data sets.

A significant component in the data sampling for dynamical processes are
the external inputs of the systems.
Classically, these are chosen to be random signals.
Such inputs satisfy with probability one the persistent excitation condition
that guarantees, for example, upper bounds on the sample complexity for
system identification~\cite{WilRMetal05}.
While this choice of inputs allows to eventually identify all information about
the system, it often is too generic to achieve small sampling numbers for
the particular task of designing stabilizing controllers for a particular given
system.
In fact, it has been shown in~\cite{WerP24} that there exist significantly
smaller data sets that are informative for the task of system
stabilization.
In this work, we propose to construct such informative data sets by iteratively
updating the input signals for the system based on a sequence of
low-dimensional stabilizing controllers over subspaces computed via
low-dimensional controller inference.
These subspaces are expanded using the sampled data until the full unstable
dynamics of the underlying system are stabilized.
The main motivation for this approach is that the sample complexity of
low-dimensional controller inference only scales with the dimension of the
subspace such that only few data samples are needed to stabilize the system in
the low-dimensional subspace, as we will show.
Besides a lower sample complexity for constructing stabilizing controllers,
the proposed adaptive data generation suppresses the dynamical
instabilities during the sampling process, ensuring informative data for
stabilization.

The remainder of this manuscript is organized as follows.
In \Cref{sec:basics}, we remind the reader of basic concepts needed for the
stabilization of dynamical systems and the data informativity framework for
learning stabilizing controllers.
We derive the adaptive data sampling scheme in \Cref{sec:method} based on
theoretical results that outline suitable input signals, and we discuss
convergence of the method as well as extensions for stable numerical
computations.
Numerical experiments are performed in \Cref{sec:experiments} including
the stabilization of a power network and the laminar flow behind an obstacle.
The paper is concluded in \Cref{sec:conclusions}.

%%%%%%%%%%%%%%%%%%%%%%%%%%%%%%%%%%%%%%%%%%%%%%%%%%%%%%%%%%%%%%%%%%%%%%%%%%%%%%%%
% PRELIMINARIES.                                                               %
%%%%%%%%%%%%%%%%%%%%%%%%%%%%%%%%%%%%%%%%%%%%%%%%%%%%%%%%%%%%%%%%%%%%%%%%%%%%%%%%

\section{Preliminaries}%
\label{sec:basics}

%%%%%%%%%%%%%%%%%%%%%%%%%%%%%%%%%%%%%%%%%%%%%%%%%%%%%%%%%%%%%%%%%%%%%%%%%%%%%%%%

\subsection{Stabilization of dynamical systems}%
\label{sec:stabilization}

Let $(\xn(t))_{t \geq 0}$ be a process that can be controlled with inputs
$(\un(t))_{t \geq 0}$, where $\xn(t) \in \R^{\nh}$ and $\un(t) \in \R^{\np}$
are state and control at time $t$, respectively.
For feasible initial conditions $\xn(0) \in \Xcal_{0}^{\operatorname{n}}
\subseteq \R^{\nh}$, the dynamics of $\xn(t)$ are described in continuous time
by systems of ordinary differential equations
\begin{equation} \label{eqn:sys_ct}
  \ddt \xn(t) = f(\xn(t), \un(t)), \quad t \geq 0,
\end{equation}
with the time evolution function
$f\colon \R^{\nh} \times \R^{\np} \to \R^{\nh}$,
which can be nonlinear in both arguments.
In discrete time, processes are governed by systems of difference equations
\begin{equation} \label{eqn:sys_dt}
  \xn(t + 1) = f(\xn(t), \un(t)), \quad t \in \N \cup \{ 0 \}.
\end{equation}
The superscript $\operatorname{n}$ used above in the notation of the
states $\xn(t)$ and inputs $\un(t)$ indicates throughout this work the
association of these quantities with the original, potentially nonlinear
system described by the model $f$.
We are interested in stabilizing processes towards their steady states.
A steady state $(\xs, \us)$ is an equilibrium of the time evolution function
such that in the continuous-time case~\cref{eqn:sys_ct} it holds
\begin{equation} \label{eqn:steady_ct}
  0 = f(\xs, \us),
\end{equation}
and for discrete-time systems~\cref{eqn:sys_dt} it holds
\begin{equation} \label{eqn:steady_dt}
  \xs = f(\xs, \us).
\end{equation}
Steady states are called unstable if trajectories do not converge to
$(\xs, \us)$ for initial conditions $\xn(0)$ in an arbitrarily small
neighborhood about $(\xs, \us)$.
The task we consider is to construct a state-feedback controller
$K \in \R^{\np \times \nh}$ such that there exists an $\epsilon > 0$ so that
for all initial conditions $\xn(0) \in \Xcal_{0}^{\operatorname{n}}$ with
$\lVert \xn(0) - \xs \rVert \leq \epsilon$, the control signal given by
\begin{equation} \label{eqn:K}
  \un(t) = K(\xn(t) - \xs) + \us
\end{equation}
leads to $\xn(t) \to \xs$ in the limit $t \to \infty$.
Such a controller $K$ is then denoted as locally stabilizing.

%%%%%%%%%%%%%%%%%%%%%%%%%%%%%%%%%%%%%%%%%%%%%%%%%%%%%%%%%%%%%%%%%%%%%%%%%%%%%%%%

\subsection{Controlling systems from data}%
\label{sec:coninf}

In this work, a model of the system dynamics in terms of the right-hand side
function $f$ in~\cref{eqn:sys_ct} or~\cref{eqn:sys_dt} is unavailable in the
sense that $f$ is not given in closed form.
Instead, we assume that $f$ can be evaluated for suitable states and control
inputs.
For either system type~\cref{eqn:sys_ct} or~\cref{eqn:sys_dt}, this leads
to data triplets $\datatripn$ of the form
\begin{subequations} \label{eqn:datatripn}
\begin{align} \label{eqn:datatripn_u}
  \Un & = \begin{bmatrix} \un(t_{1}) &
    \un(t_{2}) & \ldots &
    \un(t_{\nT})
    \end{bmatrix}, \\ \label{eqn:datatripn_x}
  \Xn & = \begin{bmatrix} \xn(t_{1}) &
    \xn(t_{2}) & \ldots &
    \xn(t_{\nT})
    \end{bmatrix}, \\ \label{eqn:datatripn_y}
  \Yn & = \begin{bmatrix}
    f(\xn(t_{1}), \un(t_{1})) &
    f(\xn(t_{2}), \un(t_{2})) & \ldots &
    f(\xn(t_{\nT}), \un(t_{\nT})) 
  \end{bmatrix},
\end{align}
\end{subequations}
at some discrete time points $t_{1}, t_{2}, \ldots, t_{\nT}$.
These time points do not need to be ordered or have a particular distance
to each other and may even be identical.

For continuous-time systems, the columns of $\Yn$ in~\cref{eqn:datatripn_y}
are the derivatives of~\cref{eqn:datatripn_x} with respect to time such that
\begin{equation*}
  \Yn = \begin{bmatrix} \ddt \xn(t_{1}) &
    \ddt \xn(t_{2}) & \ldots &
    \ddt \xn(t_{\nT}) \end{bmatrix}.
\end{equation*}
The simulation of trajectories of continuous-time systems~\cref{eqn:sys_ct}
typically yields a discretization error.
However, the considered data setup is independent of this type of error
since only evaluation pairs of the time evolution function are used and
no exact, continuous trajectories over multiple time steps are needed.
Note that if the time derivatives are only approximated using the states
of consecutive time steps, the time discretization as well as approximation
errors play a role.
However, these errors can typically be controlled by the choice of the time-step
size and discretization schemes.
We do not discuss this particular case in this work.

In the case of discrete-time systems~\cref{eqn:sys_dt}, the evaluation of the
evolution function~$f$ amounts to single forward time steps on trajectories
such that $\Yn$ in~\cref{eqn:datatripn_y} simplifies to
\begin{equation*}
  \Yn = \begin{bmatrix} \xn(t_{1} + 1) & 
    \xn(t_{2} + 1) & \ldots &
    \xn(t_{\nT} + 1) \end{bmatrix}.
\end{equation*}
A classical choice for~\cref{eqn:datatripn} in this case is single
trajectory data resulting in $\Yn$ being $\Xn$ shifted by a single time step
such that
\begin{equation*}
  \begin{aligned}
    \Un & = \begin{bmatrix} \un(0) & \un(1) & \ldots &
      \un(T - 1) \end{bmatrix}, \\
    \Xn & = \begin{bmatrix} \xn(0) & \xn(1) & \ldots &
      \xn(T - 1) \end{bmatrix}, \\
    \Yn & = \begin{bmatrix} \xn(1) & \xn(2) & \ldots &
      \xn(T) \end{bmatrix};
  \end{aligned}
\end{equation*}
see, for example,~\cite{WerP24, WerP23a, VanETetal20, DePT20, Sch10, BruPK16}.

%%%%%%%%%%%%%%%%%%%%%%%%%%%%%%%%%%%%%%%%%%%%%%%%%%%%%%%%%%%%%%%%%%%%%%%%%%%%%%%%

\subsection{Data informativity and controller inference}%
\label{sec:informinfer}

In contrast to traditional approaches, which train some chosen parametrization
of $f$ using~\cref{eqn:datatripn} to infer a model describing the system
dynamics, we are going to directly learn a state feedback controller $K$ that
stabilizes the unknown system described by $f$.
This has been shown to allow for significant reductions in terms of
data needed to solve the task~\cite{WerP24, WerP23a, VanETetal20, DePT20}.
A general concept for data-driven tasks that has been developed in recent
years is data informativity~\cite{WilRMetal05, VanECetal23}.
Underlying is the question, how much data of the form~\cref{eqn:datatripn}
is needed to achieve a task.
This circumvents the need of identifying the true underlying system dynamics
and allows to formulate sample complexity results in terms of the task
at hand.

Necessary and sufficient conditions for data informativity for the
stabilization of linear dynamical systems by state-feedback have been developed
in~\cite{VanETetal20, DePT20} and extended to low-dimensional subspaces
in~\cite{WerP24}.
A relation to nonlinear systems has been outlined in~\cite{WerP23a}, which we
will quickly recap in the following.
For analytic models $f$ in~\cref{eqn:sys_ct,eqn:sys_dt}, the local
dynamics about the steady state $(\xs, \us)$ can be described by a
linearization of the form
\begin{equation} \label{eqn:linearization}
  \begin{aligned}
    f(\xn(t), \un(t)) & = f(\xs, \us) + \Jx (\xn(t) - \xs) +
      \Ju(\un(t) - \us)\\
   & \quad{}+{}
    \Ocal(\lVert \xn(t) - \xs \rVert^{2}) +
    \Ocal(\lVert \un(t) - \us \rVert^{2}),
  \end{aligned}
\end{equation}
where the Jacobians of $f$ with respect to the states $x$ and inputs $u$
evaluated in the steady state $(\xs, \us)$ are given by
\begin{equation*}
  \Jx = \nabla_{\operatorname{x}}f(\xs, \us) \quad\text{and}\quad
  \Ju = \nabla_{\operatorname{u}}f(\xs, \us).
\end{equation*}
The deviation from the desired steady state
$\xt(t) = \xn(t) - \xs$ is then
approximated by the state $x(t)$ of the linear systems
\begin{equation} \label{eqn:linsys}
    \ddt x(t) = \Jx x(t) + \Ju u(t) \quad\text{and}\quad
    x(t + 1) = \Jx x(t) + \Ju u(t),
\end{equation}
for continues and discrete-time systems, respectively, with
$u(t) = \un(t) - \us$ and the
initial condition $x(0) \in \Xcal^{\operatorname{n}}_{0} - \xs$.
The data triplets $\datatrip$ corresponding
to~\cref{eqn:linsys} are related to~\cref{eqn:datatripn} as
approximations to columnwise shifts of the data given by
\begin{subequations} \label{eqn:shiftdata}
\begin{align}
  U & = \Un - \us = \begin{bmatrix} \un(t_{1}) - \us & \ldots &
    \un(t_{T}) - \us \end{bmatrix}, \\
  \Xt & = \Xn - \xs = \begin{bmatrix} \xn(t_{1}) - \xs & \ldots &
    \xn(t_{T}) - \xs \end{bmatrix}, \\
  \Yt & = \Yn - f(\xs, \us) \\ \nonumber
    & = \begin{bmatrix}
    f(\xn(t_{1}), \un(t_{1})) - f(\xs, \us) & \ldots &
    f(\xn(t_{T}), \un(t_{T})) - f(\xs, \us) \end{bmatrix}
\end{align}
\end{subequations}
such that
\begin{equation} \label{eqn:errdata}
  X = \Xt + \Ocal(\lVert \Xt \rVert^{2}) \quad\text{and}\quad
  Y = \Yt + \Ocal(\lVert \Yt \rVert^{2}) + \Ocal(\lVert U \rVert^{2}).
\end{equation}
The linear controller $K$ in~\cref{eqn:K} can be designed from idealized data
triplets $\datatrip$ of linear systems~\cref{eqn:linsys} via data informativity
conditions as outlined in the following proposition.

\begin{proposition}[Data informativity for
  {stabilization~\cite{VanETetal20, WerP24}}]%
  \label{prp:informdata}
  Let $\datatrip$ be a data triplet sampled from a linear state-space model.
  The data triplet is informative for stabilization by feedback if and only if
  one of the following equivalent statements hold:
  \begin{itemize}
    \item[(a)] The matrix $X$ has full row rank and there exists a right inverse
      $X^{\dagger}$ of $X$ such that $YX^{\dagger}$ has only stable eigenvalues.
      A stabilizing controller is then given by $K = U X^{\dagger}$.
    \item[(b)] There exists a matrix $\Theta \in \R^{\nT \times \nh}$ such that
      if the sampled model is continuous in time
      \begin{equation} \label{eqn:lmi_ct}
        X \Theta > 0 \quad\text{and}\quad
        Y \Theta + \Theta^{\trans} Y^{\trans} < 0
      \end{equation}
      hold, or if the sampled model is discrete in time
      \begin{equation} \label{eqn:lmi_dt}
        X \Theta = (X \Theta)^{\trans} \quad\text{and}\quad
        \begin{bmatrix} X \Theta & Y \Theta \\ (Y \Theta)^{\trans} & X \Theta
        \end{bmatrix} > 0
      \end{equation}
      hold.
      A stabilizing controller is then given by $K = U \Theta (X \Theta)^{-1}$.
  \end{itemize}
\end{proposition}

The eigenvalue criterion in part~(a) of \Cref{prp:informdata} must be treated
differently depending on the type of model that is considered similar to
part~(b).
A matrix has only stable eigenvalues if in the continuous-time case the
real part of all eigenvalues is negative, or if in the discrete-time case the
absolute value of all eigenvalues is smaller than one.

Following \Cref{prp:informdata}, we call data informative for stabilization
if there exists a controller $K$ that stabilizes all linear state-space
models given by matrices $A \in \R^{\nh \times \nh}$ and
$B \in \R^{\nh \times \np}$ that describe the data $\datatrip$ via
$Y = A X + B U$.
In other words, the data contain enough information to solve the stabilization
problem.

The necessary condition for data informativity, namely that $X$ has full
row rank, can only be satisfied if the subspace of reachable system states is
the full state space $\R^{\nh}$.
In the case that system states span only a subspace $\Vcal$ of dimension
$\nr < \nh$, the classical concept of data informativity
(\Cref{prp:informdata}) cannot be applied anymore.
A low-dimensional data informativity concept for the task of stabilization has
been developed in~\cite{WerP24}.
Thereby, the conditions in \Cref{prp:informdata} are not considered in the full
$\nh$-dimensional state space but only in the $\nr$-dimensional subspace
$\Vcal$ using a projection.

\begin{proposition}[Low-dimensional data {informativity~\cite{WerP24}}]%
  \label{prp:lowdimdata}
  Let $\datatrip$ be a data triplet sampled from a linear state-space model
  of dimension $\nh$ and let $\Vcal$ be the $\nr$-dimensional subspace spanned
  by all trajectories of the model such that for an orthogonal basis matrix
  $V$ of $\Vcal$, the state satisfies $x(t) = V \xr(t)$ for all $t \geq 0$,
  where $\xr(t) \in \R^{\nr}$ is the state of a linear $\nr$-dimensional model.
  The data $\datatrip$ are informative for stabilization over $\Vcal$ if and
  only if $\datatripr$ satisfies \Cref{prp:informdata}, with
  \begin{equation*}
    \Xr = V^{\trans} X \quad\text{and}\quad \Yr = V^{\trans} Y.
  \end{equation*}
  A stabilizing controller is then given by $K = \Kr V^{\trans}$,
  where $\Kr$ is a controller that stabilizes all linear models that describe
  the projected data $\datatripr$.
\end{proposition}

\Cref{prp:lowdimdata} allows to design a stabilizing controller $K$ for an
$\nh$-dimensional model using only $\nr$ many data samples.
The state equality in \Cref{prp:lowdimdata} is necessary to give a guarantee
for the stabilization property; however, in practice it has been observed that
already an approximate relation $x(t) \approx V \xr(t)$ for some dimension
$r < \nmin$, where $\nmin$ is the intrinsic system dimension, is sufficient to
design a stabilizing controller via \Cref{prp:lowdimdata}.
The application of the stabilization theory to nonlinear systems is
possible via system linearization based on~\cref{eqn:linearization}.
Under the assumption that the model $f$ is analytic, it holds that for
$\xn(t)$ staying close enough to the steady state $\xs$, a controller that
stabilizes the corresponding linearization~\cref{eqn:linsys} also stabilizes the
nonlinear system; see~\cite[Prop.~3.1]{WerP23a}.
With a similar argument, we can use the shifted data $\datatript$ as
approximately linear data because the relation~\cref{eqn:errdata} holds and
design a stabilizing controller from \Cref{prp:informdata,prp:lowdimdata}.
The resulting method for low-dimensional controller inference of nonlinear
systems is summarized in \Cref{alg:ci}.
The algorithm operates on the data from the nonlinear model~\cref{eqn:sys_ct}
or~\cref{eqn:sys_dt}.
We call the controller inferred by \Cref{alg:ci} stabilizing over the subspace
$\Vcal$ with basis $V$ as its projection onto the subspace via $\Kr = K V$
is stabilizing all nonlinear systems with linearizations that describe the
shifted and projected data
\begin{equation*}
  U = \Un - \us, \quad
  \Xr = V^{\trans} (\Xn - \xs) \quad \text{and} \quad
  \Yr = V^{\trans} (\Yn - f(\xs, \us)).
\end{equation*}
Note that the subspace $\Vcal$ can be chosen arbitrarily.
However, a typically suitable choice, which satisfies $x(t) \approx V \xr(t)$,
can be computed directly from the given data $\Xn - \xs$ using, for example,
the singular value decomposition.

\begin{algorithm2e}[t]
  \SetAlgoHangIndent{1pt}
  \DontPrintSemicolon
  \caption{Low-dimensional controller inference.}%
  \label{alg:ci}
  
  \KwIn{High-dimensional data triplet $\datatripn$,
    orthogonal basis matrix~$V$.}
    
  \KwOut{State-feedback controller~$K$.}

  Compute the reduced and shifted data triplet $\datatripr$ via
    \begin{equation*}
      U = \Un - \us, \quad
      \Xr = V^{\trans} (\Xn - \xs) \quad \text{and} \quad
      \Yr = V^{\trans} (\Yn - f(\xs, \us)).
    \end{equation*}\;
    \vspace{-\baselineskip}

  Infer a low-dimensional stabilizing feedback
    $\Kr = U \Theta (\Xr \Theta)^{-1}$ for $\datatripr$ solving
    either~\cref{eqn:lmi_ct} or~\cref{eqn:lmi_dt} for the unknown $\Theta$.\;

  Lift the inferred controller $\Kr$ to the high-dimensional space
    via $K = \Kr V^{\trans}$.\;
\end{algorithm2e}

%%%%%%%%%%%%%%%%%%%%%%%%%%%%%%%%%%%%%%%%%%%%%%%%%%%%%%%%%%%%%%%%%%%%%%%%%%%%%%%%

\subsection{Problem statement}%
\label{sec:problem}

A crucial point in the design of controllers from data is the quality of
the data.
If we have the possibility to influence the generation of data by exciting a
system using initial values and control signals, the question is how to
generate an informative data set for stabilization in the sense
of~\cref{eqn:datatripn}.
Classical choices from system identification~\cite{BouT23, VanD96, WilRMetal05}
such as randomized input signals to generate state trajectories are generic and
do not take the specific task of stabilization into account.
Furthermore, using random input signals provides no way to steer the collection
process and thus can lead to unbounded growth, rendering computations with the
generated data numerically infeasible, and limit cycle behavior, in which case
system states are repeated and therefore the information content of the data is
limited.

%%%%%%%%%%%%%%%%%%%%%%%%%%%%%%%%%%%%%%%%%%%%%%%%%%%%%%%%%%%%%%%%%%%%%%%%%%%%%%%%
% ICI METHOD.                                                                  %
%%%%%%%%%%%%%%%%%%%%%%%%%%%%%%%%%%%%%%%%%%%%%%%%%%%%%%%%%%%%%%%%%%%%%%%%%%%%%%%%

\section{Iterative controller inference method}%
\label{sec:method}

We present an approach for generating informative data sets.
We first derive sufficient conditions for system inputs
that lead to informative data.
We then introduce an algorithm that iteratively applies the sufficient
conditions to generate informative data, followed by an analysis of the
convergence of the corresponding algorithm.

%%%%%%%%%%%%%%%%%%%%%%%%%%%%%%%%%%%%%%%%%%%%%%%%%%%%%%%%%%%%%%%%%%%%%%%%%%%%%%%%

\subsection{Iterative controller inference}%
\label{sec:iciidea}

By leveraging low-dimensional data informativity (\Cref{prp:lowdimdata}),
we generate a sequence of controllers with order
$\nr_{1} \leq \nr_{2} \leq \dots \leq \nr_{\nT}$,
\begin{equation*}
  \Kr_{1} \in \R^{\np \times \nr_{1}}, \quad
  \Kr_{2} \in \R^{\np \times \nr_{2}}, \quad
  \ldots, \quad
  \Kr_{\nT} \in \R^{\np \times \nr_{\nT}},
\end{equation*}
which in lifted form locally stabilize the system~\cref{eqn:sys_ct}
or~\cref{eqn:sys_dt} over a sequence of nested subspaces with non-decreasing
dimensions $\nr_{1} \leq \nr_{2} \leq \dots \leq \nr_{\nT}$,
\begin{equation*}
  \Vcal_{1} \subseteq \Vcal_{2} \subseteq \dots \subseteq \Vcal_{\nT}.
\end{equation*}
When the dimension $\nr_{\nT}$ of $\Vcal_{\nT}$ reaches the intrinsic dimension
$\nmin$ of the system~\cref{eqn:sys_ct} or~\cref{eqn:sys_dt}, then the
system is stabilized in the sense of \Cref{sec:stabilization}. 

Each iteration $j = 1, \dots, \nT$ of iterative controller inference (ICI)
proceeds in two steps.
In the first step of iteration $j$, the system is queried to extend the data
triplet from the previous iteration $\datatripnjone$ given by
\begin{equation*}
  \begin{aligned}
    \Un_{j-1} & = \begin{bmatrix} \un_{0}(t_{0}) &
      \un_{1}(t_{1}) & \ldots &
      \un_{j-2}(t_{j-2}) \end{bmatrix}, \\
    \Xn_{j-1} & = \begin{bmatrix} \xn(t_{0}) & \xn(t_{1}) & \ldots &
      \xn(t_{j-2}) \end{bmatrix}, \\
    \Yn_{j-1} & = \begin{bmatrix} f(\xn(t_{0}), \un_{0}(t_{0})) &
      f(\xn(t_{1}), \un_{1}(t_{1})) & \ldots &
      f(\xn(t_{j-2}), \un_{j-2}(t_{j-2})) \end{bmatrix},
  \end{aligned}
\end{equation*}
to the current data triplet $\datatripnj$.
To extend the data triplet, we generate an input signal $u(t)$ as 
\begin{equation} \label{eqn:updateU}
  \un_{j - 1}(t) = K_{j - 1}(\xn(t) - \xs) + \us,
\end{equation}
using the controller $K_{j-1}$ from the previous iteration.
Generating an input signal as in~\cref{eqn:updateU} is motivated by a
sufficient condition for system inputs to lead to informative data, which we
derive in \Cref{sec:adaptiveinform}.
Furthermore, exciting a system with an input given by~\cref{eqn:updateU} 
ensures that the state trajectory will stay close to the steady state $\xs$ as
needed for the linearization~\cref{eqn:linsys} of the nonlinear
system~\cref{eqn:sys_ct} or~\cref{eqn:sys_dt} to remain meaningful.
Once we have the input signal $\un_{j - 1}(t)$ obtained
via~\cref{eqn:updateU}, we query the system~\cref{eqn:sys_ct}
or~\cref{eqn:sys_dt} to obtain its action at time $t_{j - 1}$,
\begin{equation} \label{eqn:sysquery}
  \yn(t_{j-1}) = f(\xn(t_{j-1}), \un_{j-1}(t_{j-1})).
\end{equation}
Then, the new data triplet $\datatripnj$ is given by the matrices
\begin{subequations} \label{eqn:updateMTX}
\begin{align} 
  \Un_{j} & = \begin{bmatrix} \Un_{j-1} & \un_{j-1}(t_{j-1}) \end{bmatrix}, \\
  \Xn_{j} & = \begin{bmatrix} \Xn_{j-1} & \xn(t_{j-1}) \end{bmatrix}, \\
  \Yn_{j} & = \begin{bmatrix} \Yn_{j-1} & \yn(t_{j-1}) \end{bmatrix}.
\end{align}
\end{subequations}
Additionally, the subspace $\Vcal_{j-1}$ is extended to $\Vcal_{j}$ by 
\begin{equation*}
  \Vcal_{j} = \Vcal_{j-1} \cup \mspan(\xn(t_{j-1}) - \xs),
\end{equation*}
except in certain special cases, which are further discussed
in \Cref{sec:ICI:Algorithm}.
To advance the state to $\xn(t_{j})$, we integrate the system in time
using $\un_{j-1}(t)$ from~\cref{eqn:updateU}.
Note that in the case of discrete-time systems~\cref{eqn:sys_dt} and
if the discrete time points for the data sampling are 
only one time step apart so that $t_{j} = t_{j-1} + 1$ holds,
the system query and integration over time are equivalent such that
$\yn(t_{j-1}) = \xn(t_{j})$.
In the first iteration $j = 1$, the initial matrices $\Un_{0}$, $\Xn_{0}$,
and $\Yn_{0}$ are empty, and there is no controller $K_{0}$ available.
Therefore, we use a normally distributed random input signal $u(t)$
to query the system at time $t_{0}$ and to integrate the state from
$\xn(t_{0})$ to $\xn(t_{1})$.
We discuss this in more in detail in \Cref{sec:ICI:Algorithm}.

In the second step of each iteration $j$ of ICI, a new low-dimensional
controller $\Kr_{j} \in \R^{\np \times \nr_{j}}$ corresponding to the subspace
$\Vcal_{j}$ is inferred with low-dimensional controller inference
(see \Cref{alg:ci}) from the current data triplet $\datatripnj$. 
The corresponding lifted controller $K_{j} = \Kr_{j} V_{j}^{\trans}$ stabilizes
the system~\cref{eqn:sys_ct} or~\cref{eqn:sys_dt} over the subspace
$\Vcal_{j}$, which means that it stabilizes all
linearizations~\cref{eqn:linsys} that describe the shifted and projected data
\begin{equation*}
  U_{j} = \Un_{j} - \us, \quad
  \Xr = V_{j}^{\trans} (\Xn_{j} - \xs) \quad \text{and} \quad
  \Yr = V_{j}^{\trans} (\Yn_{j} - f(\xs, \us)).
\end{equation*}
If the dimension of $\Vcal_{j}$ equals the intrinsic dimension $\nmin$ of the
system~\cref{eqn:sys_ct} or~\cref{eqn:sys_dt}, the algorithm terminates and
returns the controller $K_{j}$, which stabilizes the system.
Otherwise, the process is continued with iteration $j + 1$.

This iterative process can be concluded when the dimension of the
constructed subspace $\Vcal_{\nT}$ reaches the intrinsic system dimension
$\nmin$.
Typically, the intrinsic dimension $\nmin$ is unknown for general systems.
However, stagnation in the expansion of the subspace $\Vcal_{\nT}$ indicates
that no further independent data can be generated via simulation and that
$\Vcal_{\nT}$ has reached its maximum dimension $\nmin$.
For the practical use of the method in scarce data applications, we propose to
end the iteration scheme via a criterion of the form
$\lVert \xn(t_{j}) - \xs \rVert < \operatorname{tol}$ with a user-given
tolerance $\operatorname{tol} > 0$, which stops the iterations when the
system trajectories are close enough to the steady state and can be considered
as stabilized.

%%%%%%%%%%%%%%%%%%%%%%%%%%%%%%%%%%%%%%%%%%%%%%%%%%%%%%%%%%%%%%%%%%%%%%%%%%%%%%%%

\subsection{Adaptive inputs for sampling informative data}%
\label{sec:adaptiveinform}

The key step of the ICI method is the generation of the input signal
via~\cref{eqn:updateU}.
We now show under which circumstances this specific way of constructing an
input signal leads to informative data.
We begin by showing that informative data can be obtained only via non-zero
input signals, which allows us to conclude that the input signals play an
essential role in generating informative data.  
For the ease of exposition, we consider for now only systems with linear
dynamics; however, we later apply the concept of linearization to extend these
results to nonlinear systems.

\begin{lemma}%
  \label{lmm:zeroinputs}
  Let $\datatrip$ be a data triplet of a linear model~\cref{eqn:linsys} such
  that $X \in \R^{\nh \times \nT}$ has full row rank
  and the system input samples are $U = 0 \in \R^{\np \times \nT}$.
  If the underlying linear model is unstable, then the data triplet $\datatrip$
  cannot be informative for stabilization by feedback for any $\nT \in \N$.
\end{lemma}
\begin{proof}
  The number of data samples $\nT$ in $X$ must be at least the state-space
  dimension such that $\nT \geq \nh$ for $X$ to have full row rank.
  Then, there exists an $X^{\dagger}$, which is a right inverse of $X$ such
  that $X X^{\dagger} = I_{\nh}$.
  For all right inverses $X^{\dagger}$, it then holds that the product
  $Y X^{\dagger} = \Jx X X^{\dagger} = \Jx$ has the
  same unstable spectrum as the system matrix $\Jx$ for any $\nT \geq \nh$.
  Therefore, there cannot exist a right inverse $X^{\dagger}$ such that
  $Y X^{\dagger}$ has a stable spectrum and, following part~(a) of
  \Cref{prp:informdata}, the $\datatrip$ cannot be informative for stabilization
  for any $\nT \geq \nh$, which concludes the proof.
\end{proof}

\Cref{lmm:zeroinputs} implies that there are non-zero input signals that
generate informative data because otherwise stabilization and control were
generally impossible.
Additionally, the lemma provides a sufficient stability test for models by
reversing the argument: 
If data $\datatrip$ with a zero input signal $U$ is informative for
stabilization then the original model from which the data has been sampled is
stable.

It has been shown in~\cite[Cor.~5]{WerP24} that the minimal number of samples
required for an informative data set is the intrinsic dimension $\nmin$ of the 
system and it has been shown that such data sets exist.
The key step of the proof of~\cite[Cor.~5]{WerP24} is the construction of a
suitable input signal using a stabilizing controller.
We now reverse this approach, which allows us to obtain a sufficient condition
on the choice of the input signal to guarantee the construction of informative
data triplets.
This is shown in the following lemma.

\begin{lemma}%
  \label{lmm:informdata}
  Let $f$ be the model of a linear system from which the data $\datatrip$ has
  been sampled and let $\Vcal_{\min}$ be the $\nmin$-dimensional subspace of
  all model trajectories.
  Also, let $X \in \R^{\nh \times \nT}$ have the row rank $\nmin$.
  If the input signal is chosen such that $U = K X$, where
  $K$ is a controller that stabilizes the system described by $f$,
  then the data $\datatrip$ is informative for stabilization by feedback over
  $\Vcal_{\min}$.
\end{lemma}
\begin{proof}
  Since $f$ describes a linear model, there exist matrices
  $A \in \R^{\nh \times \nh}$ and $B \times \R^{\nh \times \np}$ such that
  \begin{equation*}
    Y = A X + B U
  \end{equation*}
  holds for the given data $\datatrip$.
  Choosing the input samples as $U = KX$ then yields
  \begin{equation} \label{eqn:informdata_tmp1}
    Y = A X + B K X = (A + B K) X.
  \end{equation}
  Because $X$ is assumed to have row rank $\nmin$, it follows that
  $\nT \geq \nmin$.
  Furthermore, there exists an orthogonal basis matrix
  $V \in \R^{\nh \times \nmin}$ of the subspace $\Vcal_{\min}$
  such that $x(t) = V \xr(t)$ holds for all $t \geq 0$, where
  $\xr(t)$ is the state of a linear model of dimension $\nmin$ of the form
  \begin{equation} \label{eqn:informdata_tmp4}
    \dot{\xr}(t) = \Ar \xr(t) + \Br u(t),
  \end{equation}
  with $\Ar \in \R^{\nmin \times \nmin}$ and $\Br \in \R^{\nmin \times \np}$.
  The system matrices in~\cref{eqn:informdata_tmp4} are obtained via projection
  onto $\Vcal_{\min}$ using the orthogonal basis matrix $V$ such that
  $\Ar = V^{\trans}AV$ and $\Br = V^{\trans}B$.
  Via \Cref{prp:lowdimdata}, we know that $\datatrip$ is informative
  for stabilization over $\Vcal_{\min}$ if the projected data set
  $\datatripr$, where $X = V \Xr$ and $Y = V \Yr$, is informative for
  stabilization.
  Inserting the projected data into~\cref{eqn:informdata_tmp1} yields
  \begin{equation} \label{eqn:informdata_tmp2}
    \Yr = V^{\trans} V \Yr = V^{\trans} Y = V^{\trans} (A + BK) X =
      V^{\trans} (A + BK) V \Xr = \Ar \Xr + \Br \Kr \Xr,
  \end{equation}
  where $\Ar = V^{\trans}AV$ and $\Br = V^{\trans}B$ are as above and
  $\Kr = KV$. 
  Because $X$ has row rank $\nmin$ and $\Vcal_{\min}$ contains all state
  trajectories of the model $f$, it follows that
  $\Xr \in \R^{\nmin \times \nT}$ has full row rank $\nmin$ such that
  there exists a right inverse $\Xr^{\dagger}$ of $\Xr$.
  From~\cref{eqn:informdata_tmp2} it then follows that
  \begin{equation} \label{eqn:informdata_tmp3}
    \Yr \Xr^{\dagger} = \Ar + \Br \Kr.
  \end{equation}
  Because $\Vcal_{\min}$ is the subspace of all trajectories of the model $f$,
  the basis matrix $V$ spans an invariant subspace of $A$ such that the
  spectrum of the projection is part of the spectrum of the full
  matrix~\cite{WerP24}.
  Following the assumption that $K$ is a controller that stabilizes the system
  described by the model $f$, we have that the matrix
  \begin{equation*}
    V^{\trans} (A + B K) V = \Ar + \Br \Kr
  \end{equation*}
  is stable, too.
  From~\cref{eqn:informdata_tmp3}, we then get that the product
  $\Yr \Xr^{\dagger}$ is stable.
  By applying part~(a) of \Cref{prp:informdata}, the projected data triplet
  $\datatripr$ is informative for stabilization such that from
  \Cref{prp:lowdimdata}, the result of this lemma follows.
\end{proof}

We now extend the result from \Cref{lmm:informdata} to the more general case
of nonlinear systems.

\begin{theorem}%
  \label{thm:nonlinearstab}
  Let $f$ be a model with unstable steady state $(\xs, \us)$ from which the
  data $\datatripn$ has been sampled, with the state data
  $\Xn = \begin{bmatrix} \xn(t_{1}) & \ldots & \xn(t_{\nT}) \end{bmatrix}
  \in \R^{\nh \times \nT}$ being close to the steady state such that
  $\lVert \xn(t_{j}) - \xs \rVert \leq \epsilon$ holds for some
  $\epsilon > 0$ small enough and $j = 1, \ldots, \nT$.
  Also, let the the shifted data $\Xt$ obtained from $\Xn$
  via~\cref{eqn:shiftdata} have the row rank $\nmin$, the intrinsic dimension of the system.
  If $\Un$ is chosen via~\cref{eqn:K} for a $K$ that stabilizes the
  linearization of $f$, then the data $\datatripn$ is informative for
  local stabilization by linear feedback.
\end{theorem}
\begin{proof}
  First, we observe that for $\epsilon$ small enough, the shifted data 
  triplet $\datatript$ obtained via~\cref{eqn:shiftdata} approximates
  the idealized data $\datatrip$ of the linearization~\cref{eqn:linsys}
  arbitrarily well such that following the theory in~\cite[Sec.~10.1]{NijV16},
  we may assume that the data $\datatript$ can be described by the
  linearization~\cref{eqn:linsys}.
  Then, the $\nmin$-dimensional subspace $\Vcal_{\min}$ of all trajectories
  of the linearization is spanned by the orthogonal basis matrix resulting
  from a singular value decomposition of the given data since $\Xt$ is
  assumed to have row rank $\nmin$ such that
  \begin{equation*}
    \begin{bmatrix} V_{\min} & V_{2} \end{bmatrix}
      \begin{bmatrix} \Sigma & 0 \\ 0 & 0 \end{bmatrix} W = \Xt,
  \end{equation*}
  where the basis matrix $V_{\min}$ corresponds to the $\nmin$ nonzero
  singular values in $\Sigma$, and $W$ contains the right singular vectors.
  With the choice of $\Un$ via~\cref{eqn:K} with a locally stabilizing
  controller $K$, the assumptions of \Cref{lmm:informdata} are satisfied
  for the linearization~\cref{eqn:linsys}, which means that the shifted data
  triplet $\datatript$ is informative for stabilization over $\Vcal_{\min}$.
  From~\cite[Cor.~1]{WerP23a} it follows that if $\datatript$ is informative
  for stabilization, then $\datatripn$ is informative for local stabilization
  over $\Vcal_{\min}$, which concludes the proof.
\end{proof}

The results of \Cref{thm:nonlinearstab} suggest that it is sufficient to use a
stabilizing input signal for the generation of informative data.
A stabilizing input signal can be obtained, for example, via feedback as
$u(t) = Kx(t)$, where $K$ is a stabilizing controller. 
Additionally, for the use of linearizations, the sampled data has to stay close
to the steady state of interest $(\xs, \us)$ to satisfy the assumptions of
\Cref{thm:nonlinearstab}.
Therefore, dynamical instabilities that drive the system away from the steady
state are undesired and need to be suppressed by a stabilizing control signal.
However, \Cref{lmm:informdata,thm:nonlinearstab} assume a priori access
to a stabilizing controller to generate the control signal for the trajectory
samples.
If such a controller would be available, there would be no need to construct
a controller from data.
We approach this causality dilemma via the ICI method as outlined in
\Cref{sec:iciidea}.
Thereby, we use a sequence of stabilizing controllers over increasingly
higher dimensional subspaces to iteratively update the input signal for the
next sampling step by inferring a suitable controller.

%%%%%%%%%%%%%%%%%%%%%%%%%%%%%%%%%%%%%%%%%%%%%%%%%%%%%%%%%%%%%%%%%%%%%%%%%%%%%%%%

\subsection{Convergence analysis}%
\label{sec:convergence}

Building on the results from \Cref{sec:adaptiveinform}, we now study the
convergence of the ICI method as described in \Cref{sec:iciidea}.
A bound for the maximum number of iteration steps taken by ICI follows under
some mild assumptions directly from the system identification theory.

\begin{corollary}%
  \label{cor:converge}
  Let the iteratively constructed data triplet $\datatripnj$ defined
  in~\cref{eqn:updateMTX} from the ICI method be such that the concatenated
  matrix $\begin{bmatrix} \Xt_{j}^{\trans} & U_{j}^{\trans}
  \end{bmatrix}^{\trans}$ has full column rank for all
  $j = 1, \ldots, \nmin + \np$,
  where $\Xt_{j} \in \R^{\nh \times j}$ and
  $U_{j} \in \R^{\np \times j}$ are shifted data
  matrices~\cref{eqn:shiftdata}, and the columns of $\Xn_{j}$ are
  sufficiently close to the steady state $\xs$ such that
  $\lVert \xn(t_{j}) - \xs \rVert \leq \epsilon$ for $\epsilon > 0$
  small enough.
  Then, the controller inferred by ICI from the data $\datatripnj$
  for $j = \nmin + \np$ is guaranteed to stabilize the system described by the
  model from which the data has been sampled.
\end{corollary}
\begin{proof}
  The result follows from the sampling number for system 
  identification over low-dimensional subspaces~\cite{VanD96, WerP24}
  and the rank condition on the concatenated matrix consisting of the shifted
  data.
\end{proof}

\Cref{cor:converge} states that under the condition that each generated data
sample contains new system information in terms of the concatenated matrix
rank, the ICI method terminates with a locally stabilizing controller after
$\nmin + \np$ steps.
We note that $\nmin + \np$ steps is only a worst case bound as it corresponds
to the identification of the underlying linearized system.
This worst case iteration bound in \Cref{cor:converge} can be improved under
stricter assumptions as the following theorem shows.

\begin{theorem}%
  \label{thm:converge}
  Let $\Vcal_{j} \subseteq \R^{\nh}$ be a $j$-dimensional subspace and let
  $K_{j - 1} \in \R^{\np \times \nh}$ be a state-feedback
  controller that locally stabilizes the projected system described by
  the model $V_{j} V_{j}^{\trans} f$, for an orthogonal basis $V_{j}$ of
  the subspace $\Vcal_{j}$.
  Given the data triplet $\datatripnj$ from the ICI method as
  in~\cref{eqn:updateMTX}, we assume that the shifted state samples
  $\Xt_{j} \in \R^{\nh \times j}$ obtained from
  $\Xn_{j}$ via~\cref{eqn:shiftdata} have full column rank $j$,
  the columns of $\Xn_{j}$ are sufficiently close to the steady state $\xs$
  in the sense that $\lVert \xn(t_{j}) - \xs \rVert \leq \epsilon$ for
  $\epsilon > 0$ small enough,
  and the input samples $\Un_{j}$ satisfy~\cref{eqn:K} columnwise 
  with the state samples $\Xn_{j}$ and the controller $K_{j - 1}$.
  Then, the data triplet $\datatripnj$ is informative for stabilization
  over $\Vcal_{j}$.
  If $j = \nmin$ and $\Vcal_{j}$ is so that $\xn(t) - \xs \in \Vcal_{j}$ for all
  $\xn(t)$ generated by the model $f$ with
  $\lVert \xn(t_{j}) - \xs \rVert \leq \epsilon$, then the controller inferred
  from $\datatripnj$ is guaranteed to locally stabilize the underlying system.
\end{theorem}
\begin{proof}
  From the construction of the input samples $\Un_{j}$ and the full column
  rank of the state samples $\Xn_{j}$, the conditions of
  \Cref{thm:nonlinearstab} are satisfied for the nonlinear system
  projected onto the low-dimensional subspace $\Vcal_{j}$,
  which is described by the model $V_{j} V_{j}^{\trans} f(x(t), u(t))$.
  Therefore, the data triplet $\datatripnj$ must be informative for
  stabilization over $\Vcal_{j}$.
  Since the controller inferred from the informative data set is stabilizing
  all linearizations that describe the shifted and projected data,
  it follows that if $j = \nmin$ and $\Vcal_{j}$ containing all state
  trajectories generated by $f$, the linearization of the underlying system
  is stabilized by the inferred controller.
  Therefore, the controller is guaranteed to locally stabilize the underlying
  system, which concludes the proof.
\end{proof}

Let us discuss the implications of \Cref{thm:converge}.
First, in contrast to \Cref{thm:nonlinearstab}, we consider in
\Cref{thm:converge} the use of low-dimensional subspaces $\Vcal_{j}$, which
reduce the sample complexity of the task of stabilization to the subspace
dimension.
Under the assumption that the subspace $\Vcal_{j}$ with $j = \nmin$ contains
all possible state trajectories of the sampled model around the steady state
$\xs$ and that the controller $K_{j-1}$ used for the generation of the complete
input signal $\Un_{j}$ is stabilizing over $\Vcal_{j}$, \Cref{thm:converge}
reduces the iterations of ICI from $\nmin + p$ in \Cref{cor:converge}
to $\nmin$.

We note that \Cref{thm:converge} critically relies on the assumption that
$K_{j-1}$ stabilizes over the subspace $\mathcal{V}_{j}$.
If this assumption does not hold, then our ICI method is still applicable but
it is not guaranteed to give a minimal data set.
Furthermore, considering the assumptions made in \Cref{thm:converge}, we have
to note that in the ICI method as presented in \Cref{sec:iciidea}, the
assumption that the same $K_{j-1}$ has been used for the generation of
$\Un_{j}$ will not be satisfied since the controller is updated in every
step of the method.
However, we provide a numerical extension of the method to satisfy this
assumption with further details in \Cref{sec:reproj}.
Also, we can see that the assumption of $K_{j-1}$ being stabilizing
over $\Vcal_{j}$ can be interpreted as $\Vcal_{j}$ not containing
more unstable dynamics than $\Vcal_{j-1}$.
To see this, we use from \Cref{sec:iciidea} the fact that in ICI, we have
that $\Vcal_{j-1} \subseteq \Vcal_{j}$ holds for all $j \geq 1$, and we assume
for simplicity of presentation that the corresponding basis matrices
$V_{j-1}$ and $V_{j}$ are computed via the Gram-Schmidt procedure so that
the columns of $V_{j-1}$ are the first $\nr_{j-1}$ many columns of $V_{j}$.
For the controller obtained in the previous iteration via
$K_{j-1} = \Kr_{j-1} V_{j-1}^{\trans}$, it follows that for all possible states
from the extended subspace $x \in \Vcal_{j}$ it holds
\begin{equation*}
  K_{j-1} x = K_{j-1} V_{j} \xr = \Kr_{j-1} V_{j-1}^{\trans} V_{j} \xr
    = \begin{bmatrix} \Kr_{j-1} & 0 \end{bmatrix} \xr 
    = \Kr_{j-1} \begin{bmatrix} \xr_{1} \\ \vdots \\ \xr_{j-1} \end{bmatrix}
    + 0.
\end{equation*}
This shows that the controller $K_{j-1}$ takes the same actions for states
from $\Vcal_{j}$ as for states from $\Vcal_{j-1}$, and if $\Vcal_{j}$ does
contain unstable dynamics that are not contained in $\Vcal_{j-1}$, then
$K_{j-1}$ cannot be stabilizing.
Note that if $K_{j-1}$ is not stabilizing over $\Vcal_{j}$, this does not
imply that there is no $K_{j}$ stabilizing over $\Vcal_{j}$ or that the
corresponding data triplet $\datatripnj$ is not informative for
stabilization; see the discussions in~\cite{WerP24}.
We discuss the handling of intermediate non-informative data sets are
in \Cref{sec:ICI:Algorithm}.

%%%%%%%%%%%%%%%%%%%%%%%%%%%%%%%%%%%%%%%%%%%%%%%%%%%%%%%%%%%%%%%%%%%%%%%%%%%%%%%%

\subsection{Algorithm description}%
\label{sec:ICI:Algorithm}

\begin{algorithm2e}[t]
  \SetAlgoHangIndent{1pt}
  \DontPrintSemicolon
  \caption{Iterative controller inference (ICI).}%
  \label{alg:ici}
  
  \KwIn{Queryable and integrable model $f$,
    steady state $(\xs, \us)$,
    initial value~$\xn(t_{0})$,
    time discretization scheme $t_{0} < t_{1} < \ldots < t_{\nT}$.}
  \KwOut{State-feedback controller $K_{\nT}$.}

  Initialize $U_{0} = [~]$, $X_{0} = [~]$, $Y_{0} = [~]$,
    $V_{0} = [~]$ and $K_{0} = 0$.\;

  \For{$j = 1$ \textbf{\emph{to}} $\nT$}{
    \eIf{$K_{j-1} == 0$}{
      Choose a random input signal $\un_{j-1}(t)$ with expectation
        $\E(\un_{j-1}(t)) = \us$.\;
        \label{alg:ici_randU}
    }{
      Generate the next control signal
        $\un_{j-1}(t) = K_{j-1} (\xn(t) - \xs) + \us$.\;
        \label{alg:ici_stabU}
    }

    Query the model $f$ at $\xn(t_{j-1})$ and $\un_{j-1}(t_{j-1})$ to
      obtain $\yn(t_{j-1}) = f(\xn(t_{j-1}),
      \un_{j-1}(t_{j-1}))$.\;
      \label{alg:ici_query}

    Integrate the model $f$ from $t_{j-1}$ to $t_{j}$ starting at
      $\xn(t_{j-1})$ and using the input signal~$\un_{j-1}(t)$ to obtain
      $\xn(t_{j})$.\;
      \label{alg:ici_step}

  \eIf{$\Vcal_{j-1} == \{ 0 \}$ \textbf{\emph{or}}
    $K_{j-1}$ stabilizing over $\Vcal_{j-1}$}{
    Update the projection space $\Vcal_{j}$ via
      $V_{j} = \orth(V_{j-1}, \xn(t_{j-1}) - \xs)$.\;
      \label{alg:ici_extend}
  }{
    Keep $V_{j} = V_{j-1}$.\;
      \label{alg:ici_keep}
  }

    Update the data triplet $\datatripnjone$ as
      \begin{equation*}
        \begin{aligned}
          \Un_{j} & = \begin{bmatrix} \Un_{j-1} & \un_{j-1}(t_{j-1})
            \end{bmatrix}, \\
          \Xn_{j} & = \begin{bmatrix} \Xn_{j-1} & \xn(t_{j-1}) \end{bmatrix}, \\
          \Yn_{j} & = \begin{bmatrix} \Yn_{j-1} & \yn(t_{j-1}) \end{bmatrix}.
        \end{aligned}
      \end{equation*}\;
      \label{alg:ici_data}
      \vspace{-\baselineskip}

    Infer the next controller $K_{j}$ from $\datatripnj$ and $V_{j}$
      using \Cref{alg:ci}.\;
      \label{alg:ici_ci}

    \If{no stabilizing $K_{j}$ found}{
      Overwrite data $\datatripnj$ via \Cref{alg:reproj} and
        update~$K_{j}$.\;
      \label{alg:ici_reproj}
    }
  }
\end{algorithm2e}

The complete \emph{Iterative Controller Inference (ICI)} method is summarized
in \Cref{alg:ici} following the concepts introduced in the previous sections.

In the first steps of the iteration, there is no stabilizing controller
known yet.
Thus, we propose to generate the first data samples using random inputs
$\un(t)$.
These inputs should still be centered about the steady state control
$\us$ to generate meaningful data such that we assume that in expectation it
holds $\E(\un(t)) = \us$.
Furthermore, in the first few iterations, there is no guarantee that the
current data triplet $\datatripnj$ is informative for stabilization over the
current subspace $\Vcal_{j}$.
In this case, we keep the subspace dimension constant over multiple iterations
until the data triplet becomes informative.
This happens eventually after $\np$ additional steps under the assumption that
the state samples are linearly independent since with a subspace of dimension
$\nr_{j}$ and $\nr_{j} + \np$ data samples, we can identify a unique linear
model in $\Vcal_{j}$ that describes the given data using system
identification~\cite{VanD96}.

To start the ICI method, an arbitrary initial condition
$x(t_{0}) \in \Xcal_{0}^{\operatorname{n}}$ is sufficient as long as it is
close enough to the steady state $\xs$ in the sense that
$\lVert x(t_{0}) - \xs \rVert$ is small enough such that the underlying system
can be described well by its linearization~\cref{eqn:linsys}.
A recommended choice is the steady state itself $\xn(t_{0}) = \xs$.

One motivation for the use of ICI is that trajectories constructed by the
method are expected to remain in the surrounding of the steady state $\xs$
in the sense that $\lVert x(t_{j}) - \xs \rVert$ is small for
$j = 0, \ldots, \nT$.
This property is also beneficial for other numerical problems that are
solved via linearizations.
The following remark discusses the use of \Cref{alg:ici} for general system
identification.

\begin{remark}%
  \label{rmk:sysid}
  Data generated via a single linear stabilizing controller does not allow for
  linear system identification since
  \begin{equation*}
    \mrank(X) = \mrank\left(\begin{bmatrix} X \\ K X \end{bmatrix}\right)
      \leq \nh < \nh + \np
  \end{equation*}
  holds for any data matrix $X$ and fixed feedback $K$; see~\cite{VanD96}.
  However, the data generated by \Cref{alg:ici} is typically well suited
  nevertheless for system identification.
  First, only $\np$ additional iteration steps without subspace extension
  are needed to satisfy the rank conditions of the data matrices over the low
  dimensional subspace.
  These additional steps need to generate independent data, which can typically
  be obtained using random input signals.
  Second, as for the construction of controllers based on linearization theory,
  data must be generated close to the steady state of interest to yield accurate
  results in linear system identification.
  Therefore, the use of stabilizing controllers in the data generation results
  in the desired behavior of trajectories to remain in the vicinity of the
  steady state.
\end{remark}

The ICI method as shown in \Cref{alg:ici} contains a numerical extension in
Line~\ref{alg:ici_reproj} that aims to improve the numerical stability of the
method and allows for the constructed data triplets to satisfy the assumptions
made in \Cref{thm:converge}.
This will be discussed in the following section.

%%%%%%%%%%%%%%%%%%%%%%%%%%%%%%%%%%%%%%%%%%%%%%%%%%%%%%%%%%%%%%%%%%%%%%%%%%%%%%%%

\subsection{Numerical stability and re-projections}%
\label{sec:reproj}

\begin{algorithm2e}[t]
  \SetAlgoHangIndent{1pt}
  \DontPrintSemicolon
  \caption{Data collection via re-projection with informed controls.}%
  \label{alg:reproj}
  
  \KwIn{Basis matrix $V = \begin{bmatrix} v_{1} & \ldots & v_{r} \end{bmatrix}
    \in \R^{\nh \times \nr}$,
    steady state $(\xs, \us)$,
    queryable system $f\colon \R^{\nh} \times \R^{\np} \to \R^{\nh}$,
    state scaling $\alpha > 0$,
    number of data samples to generate $\nT$,
    optional controller $K \in \R^{\np \times \nh}$.}
    
  \KwOut{Re-projected data triplet $\datatripn$.}

  Initialize $\Un = [~]$, $\Xn = [~]$, $\Yn = [~]$ and $k = 1$.\;

  \While{$k \leq \nT$}{
    \eIf{$k \leq r$}{
      Normalize $\displaystyle
        x = \alpha \frac{v_{k}}{\lVert v_{k} \rVert_{2}}$.\;
        \label{alg:reproj_basis}
    }{
      Compute $\displaystyle \xt = \sum_{j = 1}^{r} \beta_{j} v_{j}$,
        with random coefficients $\beta_{j}$.\;
        \label{alg:reproj_rnd}

      Normalize $\displaystyle x = \alpha \frac{\xt}{\lVert \xt \rVert_{2}}$.\;
    }

    \eIf{controller $K$ is given}{
      Compute the input signal $\un = K x + \us$.\;
        \label{alg:reproj_input}
    }{
      Choose a random input signal $\un$ with expectation $\E(\un) = \us$.\;
        \label{alg:reproj_input2}
    }

    Query the system $\yn = f(x + \xs, \un)$.\;

    Update the data matrices
    \begin{equation*}
      \Un = \begin{bmatrix} \Un & \un \end{bmatrix}, \quad
      \Xn = \begin{bmatrix} \Xn & x + \xs \end{bmatrix}, \quad
      \Yn = \begin{bmatrix} \Yn & \yn \end{bmatrix}.
    \end{equation*}\;
    \label{alg:reproj_datatrip}
    \vspace{-\baselineskip}

    \If{$k + 1 \geq r$ \textbf{and}
      stabilizing $K$ can be inferred from $\datatripn$%
      \label{alg:reproj_test}}
    {\textbf{break}}

    Increment $k \gets k + 1$.\;
  }
\end{algorithm2e}

It has been observed in~\cite{WerP24} that controller inference (\Cref{alg:ci})
can become numerically unstable with increasing amounts of data due to the poor
conditioning of the data matrices obtained from unstable systems.
The proposed remedy to the numerical instability is the use of a
re-projection scheme, which improves the quality of the data used for controller
inference.
\Cref{alg:reproj} shows such a re-projection scheme in the context of our
adaptive sampling method for nonlinear dynamical systems.

The general goal of \Cref{alg:reproj} is the construction of well-conditioned
data triplets $\datatripn$ for the use in ICI (\Cref{alg:ici}).
To this end, for a given basis matrix $V \in \R^{\nh \times \nr}$ of an
$\nr$-dimensional subspace $\Vcal$ and a queryable function
$f\colon \R^{\nh} \times \R^{\np} \to \R^{\nh}$, the data triplet $\datatripn$
is computed so that the first $\nr$ columns of $\Xn$ are linearly independent
and that all columns of $\Xn - \xs$ lie in $\Vcal$.
Additionally, if $V$ is an orthogonal basis, we have that for the projected
data it holds that
\begin{equation} \label{eqn:reprojdata}
  V^{\trans} (\Xn - \xs) = \begin{bmatrix} \alpha I_{r} & \Xt_{2} \end{bmatrix}.
\end{equation}
The scaling parameter $\alpha > 0$ is used to control the distance of the
data samples $x$ from the given steady state $\xs$.
This is needed for the generated data to satisfy the assumption used in
linearization theory, namely that $\lVert x \rVert = \lVert \xn - \xs \rVert$
is small, for $\xn$ being the state of the nonlinear model described by $f$.
When $\datatripn$ is used in \Cref{alg:ci}, then the projected
data~\cref{eqn:reprojdata} appears in the linear matrix inequalities from
\Cref{prp:informdata} and it has been observed in~\cite{WerP24} that the
identity matrix $I_{r}$ in the first block of~\cref{eqn:reprojdata} is
numerically beneficial.

In the case that a controller $K$ is given, \Cref{alg:reproj} uses this
controller to generate the input sequence in $\Un$ and to sample the data for
$\Yn$.
This yields further advantages and is important in the context of ICI.
First, using the controller $K$ to generate the columns of $\Un$ allows us
to satisfy the conditions of \Cref{thm:converge} to provide informative data
triplets.
Second, if the given controller $K$ is stabilizing over the subspace $\Vcal$
corresponding to the basis matrix $V$, the generated data samples in $\Yn$
will be either close to the steady state $\xs$ in the discrete-time case or
will be small in norm in the continuous-time case.
Therefore, the data in $\Yn$ is suitable for linearization-based methods such
as controller inference.
Finally, the use of a controller $K$ in \Cref{alg:reproj} leads to the
following result, which allows to potentially reduce the number of samples in
the generated data triplet.

\begin{theorem}%
  \label{thm:reproj}
  Let $f$ be a model with the steady state $(\xs, \us)$, and let $\Vcal$
  be an $\nr$-dimensional subspace with orthogonal basis matrix
  $V \in \R^{\nh \times \nr}$ over which $K = \Kr V^{\trans}$ is
  a locally stabilizing controller.
  If the data triplet $\datatripn$ is such that for the columns of $\Xn$ we
  have $\xn_{j} = \alpha v_{j} + \xs$ with $\alpha> 0$ so that
  $\lVert \xn_{j} - \xs \rVert \leq \epsilon$ holds for some $\epsilon > 0$
  small enough and $j = 1, \ldots, \nT$, and the columns of $\Un$ are obtained
  via~\cref{eqn:K}, then $\datatripn$ is an informative data set for
  stabilization over $\Vcal$ with $\nT = \nr$ many data samples.
\end{theorem}
\begin{proof}
  By construction of the data triplet, we have that the row rank of
  $\Xn - \xs$ is $\nr$.
  Also, for all linear systems, whose dynamics evolve in $\Vcal$, we have
  that $\nmin \leq \nr$ holds for the intrinsic system dimension.
  Thus, using the assumptions of this theorem, it holds that
  $K V = \Kr$ is a stabilizing controller for all linear systems
  whose dynamics evolve in $\Vcal$ and that are consistent with the given shifted
  data.
  Therefore, we can apply \Cref{thm:nonlinearstab} to the systems restricted
  to the subspace $\Vcal$, which proves that $\datatripn$ is informative
  for local stabilization of $f$ over $\Vcal$ with only $\nT = \nr$
  many data samples.
\end{proof}

\Cref{thm:reproj} states that if a controller, which stabilizes the underlying
system over $\Vcal$, has been given to the re-projection method in
\Cref{alg:reproj}, the re-projected data triplet will be informative for
stabilization over $\Vcal$ already at step $\nr$ of the re-projection scheme.
In other words, the number of data samples in the re-projected data set can
be reduced from $\nT \geq \nr$ to $\nr$ while preserving their informativity.
This is expected to increase the numerical stability as well as to reduce the
computational costs of the ICI method.
However, in particular in the presence of numerical errors, the application of
more than $\nr$ re-projection steps may be desired to further numerically
stabilize the low-dimensional controller inference.
Therefore, in Line~\ref{alg:reproj_rnd} of \Cref{alg:reproj} additional data
may be generated from random samples in $\Vcal$.

The re-projection method from \Cref{alg:reproj} is incorporated into
ICI (\Cref{alg:ici}) as a fall back if no stabilizing controller could be
inferred from the current data triplet.
This is inspired by the observation that in many cases when numerical
instabilities cause the controller inference via linear matrix inequalities to
fail computationally, the data quality improvements via re-projection
can resolve this issue.
Also, note that in this setup, re-projection is not applied in every iteration
step, which keeps the number of overall model queries small.
To further improve efficiency, we have added a testing procedure
to \Cref{alg:reproj} in Line~\ref{alg:reproj_test} that
employs the theory of \Cref{thm:reproj}.
If the number of requested data samples $\nT$ in \Cref{alg:reproj} is larger
than the current subspace dimension $\nr$, we apply the low-dimensional
controller inference method (\Cref{alg:ci}) after the $\nr$-th step of
\Cref{alg:reproj} to test if there exists a stabilizing controller $K$ over
$\Vcal$ and to terminate the re-projection scheme early.
In this case, the size of the data triplet is reduced from $\nT$ to $\nr$
for the next iteration of~ICI.

%%%%%%%%%%%%%%%%%%%%%%%%%%%%%%%%%%%%%%%%%%%%%%%%%%%%%%%%%%%%%%%%%%%%%%%%%%%%%%%%
% EXPERIMENTS.                                                                 %
%%%%%%%%%%%%%%%%%%%%%%%%%%%%%%%%%%%%%%%%%%%%%%%%%%%%%%%%%%%%%%%%%%%%%%%%%%%%%%%%

\section{Numerical experiments}%
\label{sec:experiments}

We demonstrate the proposed ICI method on a series of numerical examples.
The experiments reported in the following have been run on a machine equipped
with an Intel(R) Core(TM) i7-8700 CPU at 3.20\,GHz and with 16\,GB main memory.
The algorithms are implemented in MATLAB 9.9.0.1467703 (R2020b) on
CentOS Linux release 7.9.2009 (Core).
For the solution of the linear matrix inequalities, the disciplined convex
programming toolbox CVX version~2.2, build~1148 (62bfcca)~\cite{GraB08, GraB20}
is used together with MOSEK version 9.1.9~\cite{MOS19} as inner optimizer.
The source codes, data and results of the numerical experiments are
open source/open access and available at~\cite{supWer25}.

%%%%%%%%%%%%%%%%%%%%%%%%%%%%%%%%%%%%%%%%%%%%%%%%%%%%%%%%%%%%%%%%%%%%%%%%%%%%%%%%

\subsection{Experimental setup and summary of results}%
\label{sec:exsetup}

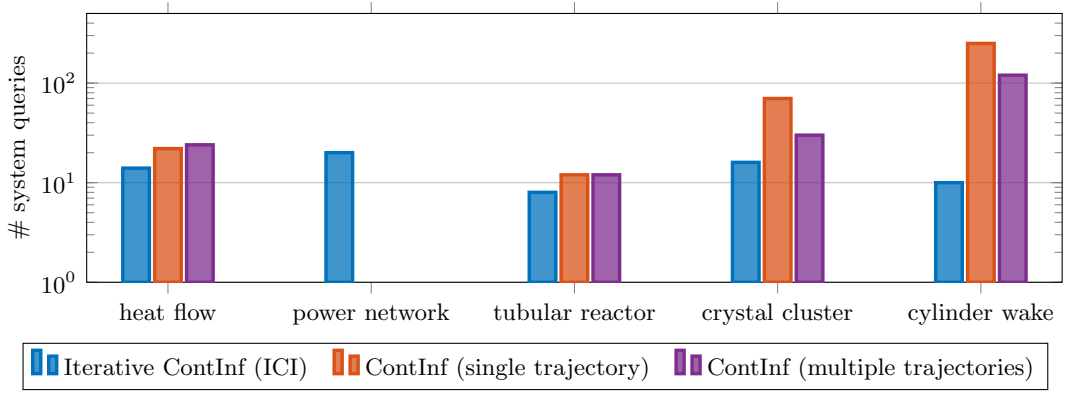
\begin{figure}
  \centering
  \tikzexternalenable%
  \tikzsetnextfilename{overview_samples_dt}%
  \begin{tikzpicture}[font = \plotfontsize]
  \begin{axis}[
    width  = .86\textwidth,
    height = .15\textheight,
    scale only axis,
    ybar,
    bar width    = 10pt,
    ymode        = log,
    ymin         = 1,
    ymax         = 5e+02,
    ylabel       = {\#~system queries},
    ymajorgrids,
    ylabel style = {yshift = -.3em},
    xtick             = data,
    symbolic x coords = {heat~flow, power~network,
      tubular~reactor, crystal~cluster, cylinder~wake},
    % xticklabel style  = {rotate = 12, anchor = north, xshift = -1.4em},
    legend columns    = 3,
    legend cell align = {left},
    legend style      = {
      at     = {(.46,-0.225)},
      anchor = north,
      /tikz/every even column/.append style = {column sep = 0.25cm}},
    cycle list name = samplelist
  ]
  
    % \addplot+[ybar] coordinates{
    %   (heat~flow, 4491)
    %   (power~network, 7139)
    %   (tubular~reactor, 4000)
    %   (crystal~cluster, 4003)
    %   (cylinder~wake, 6624)
    % };
    % \addlegendentry{traditional with system identification}
  
    \addplot+[ybar] coordinates{
      (heat~flow, 14)
      (power~network, 20)
      (tubular~reactor, 8)
      (crystal~cluster, 16)
      (cylinder~wake, 10)
    };
    \addlegendentry{Iterative ContInf (ICI)}
  
    \addplot+[ybar] coordinates{
      (heat~flow, 22)
      (power~network, NaN)
      (tubular~reactor, 12)
      (crystal~cluster, 70)
      (cylinder~wake, 250)
    };
    \addlegendentry{ContInf (single trajectory)}
  
    \addplot+[ybar] coordinates{
      (heat~flow, 24)
      (power~network, NaN)
      (tubular~reactor, 12)
      (crystal~cluster, 30)
      (cylinder~wake, 120)
    };
    \addlegendentry{ContInf (multiple trajectories)}
  \end{axis}
\end{tikzpicture}%
  \tikzexternaldisable%

  \caption{
    In all experiments, the new iterative controller inference (ICI) method
    needs less system queries to generate data sets informative enough for
    stabilization of the underlying system, with improvement factors up to $25$.
    In the case of the power network example, we were not able to find a setup
    in which classical low-dimensional controller inference (ContInf) with
    unguided data sampling was able to stabilize the system.}
  \label{fig:overview_samples}
\end{figure}

We compare our proposed ICI method from \Cref{alg:ici} to the non-adaptive
low-di\-men\-sion\-al controller inference (ContInf) from \Cref{alg:ci}
using either data samples from a single trajectory or multiple
trajectories, where the simulations have been restarted.
In either case, random input signals drawn from a normal distribution with
expectation $\E(\un(t)) = \us$ have been used to excite the systems for the
data generation.
Due to numerical instabilities using the generated raw data for controller
inference, also the non-adaptive approaches have been equipped with
the re-projection scheme \Cref{alg:reproj}.

An overview about the different experiments for discrete-time systems is shown
in \Cref{fig:overview_samples}.
The reported numbers of system queries include, when necessary, applications
of the re-projection scheme (\Cref{alg:reproj}) for any of the methods.
In all experiments, the new approach needed less system queries than
the state-of-the-art methods by factors ranging from $1.5$ to $25$.
More details are given in the following sections describing the individual
experiments.
For additional experiments with continuous-time systems, the adaptive solvers
from MATLAB's ODE Suite~\cite{ShaR97} have been used for the time integration.
As these do not allow for one-to-one comparisons in terms of system evaluations
due to the adaptive time stepping of the integration methods, only results for
the ICI method are shown in the following.

%%%%%%%%%%%%%%%%%%%%%%%%%%%%%%%%%%%%%%%%%%%%%%%%%%%%%%%%%%%%%%%%%%%%%%%%%%%%%%%%

\subsection{Unstable heat flow}%
\label{sec:heatflow}

The unstable heat flow example models the effects of overheating in
machine tools.
We consider here the \textsf{HF2D5} model from~\cite{Lei04}, in which the
heat equation is discretized in a two-dimensional rectangular domain.
Disturbances are applied to the Laplace operator leading to instabilities.
The example is described by a model of $\nh = 4\,489$ linear ODEs with
$\np = 2$ inputs.
The desired steady state is considered to be $(\xs, \us) = (0, 0)$ and the
test simulations are performed with a constant input signal.

\begin{figure}[t]
  \centering
  \begin{subfigure}[b]{.49\linewidth}
    \centering
  \tikzexternalenable%
  \tikzsetnextfilename{heatflow_ct_nofb}%
  \begin{tikzpicture}[font = \plotfontsize]
  \pgfplotstableread{graphics/data/hf2d5_ct_sim_nofb.dat}\tableSIM
  
  \begin{axis}[%
    name   = states,
    width  = .73\textwidth,
    height = .1\textheight,
    scale only axis,
    xmin = 0,
    xmax = 50,
    ymin = -4e+06,
    ymax = 1e+05,
    xminorticks = false,
    yminorticks = false,
    xlabel = {time $t$},
    ylabel = {outputs $y(t)$},
    ylabel style   = {yshift = -.3em},
    scaled x ticks = false,
    x tick label style = {/pgf/number format/1000 sep={\,}},
    y tick label style = {/pgf/number format/1000 sep={\,}},
    cycle list name    = stateslist
  ]
  
    \foreach \y in {3, 4, ..., 6}{
      \addplot+ table[x index = 0, y index = \y] {\tableSIM};
    }
  \end{axis}
\end{tikzpicture}%
  \tikzexternaldisable%

    \caption{uncontrolled simulation}
    \label{fig:heatflow_ct_nofb}
  \end{subfigure}%
  \hfill%
  \begin{subfigure}[b]{.49\linewidth}
    \centering
  \tikzexternalenable%
  \tikzsetnextfilename{heatflow_ct_ici}%
  \begin{tikzpicture}[font = \plotfontsize]
  \pgfplotstableread{graphics/data/hf2d5_ct_sim_ici.dat}\tableSIM
  
  \begin{axis}[%
    name   = states,
    width  = .73\textwidth,
    height = .1\textheight,
    scale only axis,
    xmin = 0,
    xmax = 50,
    ymin = -15,
    ymax = 15,
    xminorticks = false,
    yminorticks = false,
    xlabel = {time $t$},
    ylabel = {outputs $y(t)$},
    ylabel style   = {yshift = -.3em},
    scaled x ticks = false,
    x tick label style = {/pgf/number format/1000 sep={\,}},
    y tick label style = {/pgf/number format/1000 sep={\,}},
    cycle list name    = stateslist
  ]
  
    \foreach \y in {3, 4, ..., 6}{
      \addplot+ table[x index = 0, y index = \y] {\tableSIM};
    }
  \end{axis}
\end{tikzpicture}%
  \tikzexternaldisable%

    \caption{ICI controller}
    \label{fig:heatflow_ct_ici}
  \end{subfigure}

  \vspace{0\baselineskip}
  \tikzexternalenable%
  \tikzsetnextfilename{heatflow_legend}%
  \begin{tikzpicture}[font = \plotfontsize]
  \begin{axis}[%
    hide axis,
    width  = 1cm,
    height = 1cm,
    scale only axis,
    xmin = 0,
    xmax = 10,
    ymin = 0.5,
    ymax = 1.5,
    legend columns    = -1,
    legend cell align = {left},
    legend style      = {
      at     = {(0,0)},
      anchor = center,
      /tikz/every even column/.append style = {column sep = 0.5cm}}
  ]
    
    \pgfplotsset{cycle list name = stateslist}
    \pgfplotsinvokeforeach{1, 2, ..., 4}{\addplot coordinates {(0,0)};}
    \addlegendentry{output $y_{1}$}
    \addlegendentry{output $y_{2}$}
    \addlegendentry{output $y_{3}$}
    \addlegendentry{output $y_{4}$}
  \end{axis}
\end{tikzpicture}%
  \tikzexternaldisable%

  \caption{Continuous-time unstable heat flow:
    ICI uses $28$ data samples at discrete-time points to construct a
    controller that quickly stabilizes the simulation of the unstable heat
    flow.}
  \label{fig:heatflow_ct}
\end{figure}
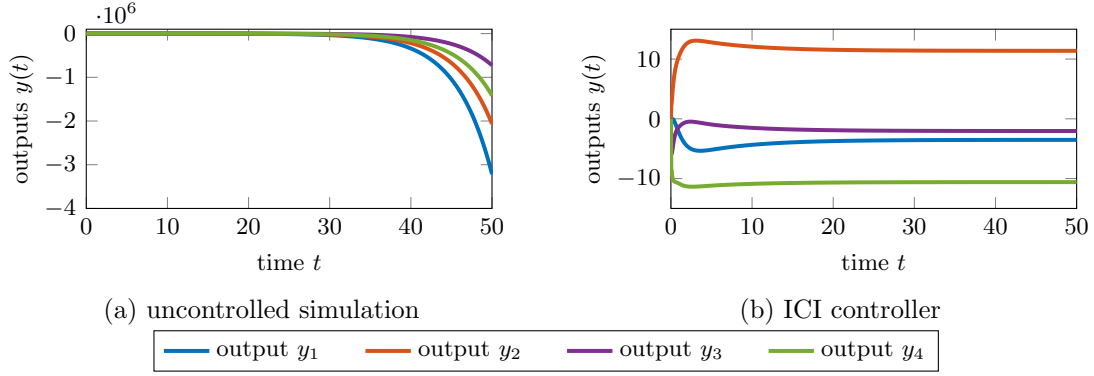

The results for the continuous-time version of the system can be seen in
\Cref{fig:heatflow_ct}.
The ICI method uses $28$ samples generated via the adaptive time integration
routine \texttt{ode45} leading overall to around $4\,024$ model evaluations.
The final simulations of the uncontrolled and stabilized system in
\Cref{fig:heatflow_ct} are computed using the \texttt{ode15s} solver
to illustrate the independence of the computed controller from the chosen time
discretization scheme in the continuous-time case.

\begin{figure}[t]
  \centering
  \begin{subfigure}[b]{.49\linewidth}
    \raggedleft
  \tikzexternalenable%
  \tikzsetnextfilename{heatflow_dt_nofb}%
  \begin{tikzpicture}[font = \plotfontsize]
  \pgfplotstableread{graphics/data/hf2d5_dt_sim_nofb.dat}\tableSIM
  
  \begin{axis}[%
    name   = states,
    width  = .73\textwidth,
    height = .1\textheight,
    scale only axis,
    xmin = 0,
    xmax = 50,
    ymin = -4e+06,
    ymax = 1e+05,
    xminorticks = false,
    yminorticks = false,
    xlabel = {time $\tau \cdot t$},
    ylabel = {outputs $y(\tau \cdot t)$},
    ylabel style   = {yshift = -.3em},
    scaled x ticks = false,
    x tick label style = {/pgf/number format/1000 sep={\,}},
    y tick label style = {/pgf/number format/1000 sep={\,}},
    cycle list name    = stateslist
  ]
  
    \foreach \y in {3, 4, ..., 6}{
      \addplot+ table[x index = 0, y index = \y] {\tableSIM};
    }
  \end{axis}
\end{tikzpicture}%
  \tikzexternaldisable%

    \caption{uncontrolled simulation}
    \label{fig:heatflow_dt_nofb}
  \end{subfigure}%
  \hfill%
  \begin{subfigure}[b]{.49\linewidth}
    \raggedleft
  \tikzexternalenable%
  \tikzsetnextfilename{heatflow_dt_ici}%
  \begin{tikzpicture}[font = \plotfontsize]
  \pgfplotstableread{graphics/data/hf2d5_dt_sim_ici.dat}\tableSIM
  
  \begin{axis}[%
    name   = states,
    width  = .73\textwidth,
    height = .1\textheight,
    scale only axis,
    xmin = 0,
    xmax = 50,
    ymin = -40,
    ymax = 30,
    xminorticks = false,
    yminorticks = false,
    xlabel = {time $\tau \cdot t$},
    ylabel = {outputs $y(\tau \cdot t)$},
    ylabel style   = {yshift = -.3em},
    scaled x ticks = false,
    x tick label style = {/pgf/number format/1000 sep={\,}},
    y tick label style = {/pgf/number format/1000 sep={\,}},
    cycle list name    = stateslist
  ]
  
    \foreach \y in {3, 4, ..., 6}{
      \addplot+ table[x index = 0, y index = \y] {\tableSIM};
    }
  \end{axis}
\end{tikzpicture}%
  \tikzexternaldisable%

    \caption{ICI controller}
    \label{fig:heatflow_dt_ici}
  \end{subfigure}
  \vspace{0\baselineskip}

  \begin{subfigure}[b]{.49\linewidth}
    \raggedleft
  \tikzexternalenable%
  \tikzsetnextfilename{heatflow_dt_sti}%
  \begin{tikzpicture}[font = \plotfontsize]
  \pgfplotstableread{graphics/data/hf2d5_dt_sim_sti.dat}\tableSIM
  
  \begin{axis}[%
    name   = states,
    width  = .73\textwidth,
    height = .1\textheight,
    scale only axis,
    xmin = 0,
    xmax = 50,
    ymin = -1000,
    ymax = 500,
    xminorticks = false,
    yminorticks = false,
    xlabel = {time $\tau \cdot t$},
    ylabel = {outputs $y(\tau \cdot t)$},
    ylabel style   = {yshift = -.6em},
    scaled x ticks = false,
    x tick label style = {/pgf/number format/1000 sep={\,}},
    y tick label style = {/pgf/number format/1000 sep={\,}},
    cycle list name    = stateslist
  ]
  
    \foreach \y in {3, 4, ..., 6}{
      \addplot+ table[x index = 0, y index = \y] {\tableSIM};
    }
  \end{axis}
\end{tikzpicture}%
  \tikzexternaldisable%

    \caption{ContInf (single trajectory)}
    \label{fig:heatflow_dt_sti}
  \end{subfigure}%
  \hfill%
  \begin{subfigure}[b]{.49\linewidth}
    \raggedleft
  \tikzexternalenable%
  \tikzsetnextfilename{heatflow_dt_mti}%
  \begin{tikzpicture}[font = \plotfontsize]
  \pgfplotstableread{graphics/data/hf2d5_dt_sim_mti.dat}\tableSIM
  
  \begin{axis}[%
    name   = states,
    width  = .73\textwidth,
    height = .1\textheight,
    scale only axis,
    xmin = 0,
    xmax = 50,
    ymin = -2e+5,
    ymax = 4e+6,
    xminorticks = false,
    yminorticks = false,
    xlabel = {time $\tau \cdot t$},
    ylabel = {outputs $y(\tau \cdot t)$},
    ylabel style   = {yshift = -.4em},
    scaled x ticks = false,
    x tick label style = {/pgf/number format/1000 sep={\,}},
    y tick label style = {/pgf/number format/1000 sep={\,}},
    cycle list name    = stateslist
  ]
  
    \foreach \y in {3, 4, ..., 6}{
      \addplot+ table[x index = 0, y index = \y] {\tableSIM};
    }
  \end{axis}
\end{tikzpicture}%
  \tikzexternaldisable%

    \caption{ContInf (multiple trajectories)}
    \label{fig:heatflow_dt_mti}
  \end{subfigure}

  \vspace{0\baselineskip}
  \tikzexternalenable%
  \tikzsetnextfilename{heatflow_legend}%
  \begin{tikzpicture}[font = \plotfontsize]
  \begin{axis}[%
    hide axis,
    width  = 1cm,
    height = 1cm,
    scale only axis,
    xmin = 0,
    xmax = 10,
    ymin = 0.5,
    ymax = 1.5,
    legend columns    = -1,
    legend cell align = {left},
    legend style      = {
      at     = {(0,0)},
      anchor = center,
      /tikz/every even column/.append style = {column sep = 0.5cm}}
  ]
    
    \pgfplotsset{cycle list name = stateslist}
    \pgfplotsinvokeforeach{1, 2, ..., 4}{\addplot coordinates {(0,0)};}
    \addlegendentry{output $y_{1}$}
    \addlegendentry{output $y_{2}$}
    \addlegendentry{output $y_{3}$}
    \addlegendentry{output $y_{4}$}
  \end{axis}
\end{tikzpicture}%
  \tikzexternaldisable%

  \caption{Discrete-time unstable heat flow:
    The new ICI approach stabilizes the system after only $14$ system
    evaluations, which is around $1.6$ times less evaluations needed by either
    of the non-adaptive approaches.
    ContInf from a single or four trajectories using
    comparable numbers of system evaluations fail to stabilize the system.}
  \label{fig:heatflow_dt}
\end{figure}
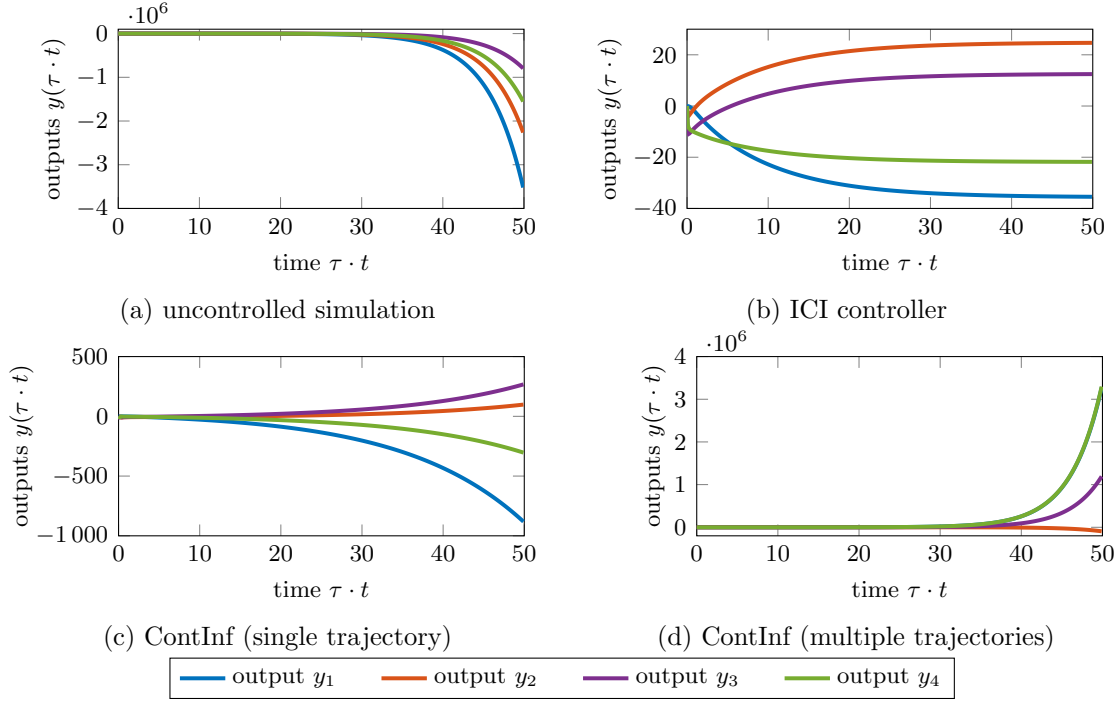

\begin{figure}[t]
  \centering
  \begin{subfigure}[b]{.49\linewidth}
    \centering
  \tikzexternalenable%
  \tikzsetnextfilename{heatflow_convergence_1}%
  \begin{tikzpicture}[font = \plotfontsize]
  \pgfplotstableread{graphics/data/hf2d5_dt_convergence.dat}\tableSIM
  
  \begin{semilogyaxis}[%
    name   = states,
    width  = .75\textwidth,
    height = .115\textheight,
    scale only axis,
    xmin = 0,
    xmax = 50,
    ymin = 1e-17,
    ymax = 5e+1,
    xminorticks = false,
    yminorticks = false,
    xlabel = {iteration steps},
    ylabel = {},
    ylabel style   = {yshift = -.3em},
    scaled x ticks = false,
    x tick label style = {/pgf/number format/1000 sep={\,}},
    y tick label style = {/pgf/number format/1000 sep={\,}},
    cycle list name    = convergencelist
  ]
  
    \foreach \y in {1, 2, 3}{
      \addplot+ table[x index = 0, y index = \y] {\tableSIM};
    }
  \end{semilogyaxis}
\end{tikzpicture}%
  \tikzexternaldisable%

    \caption{small orthogonalization tolerance}
    \label{fig:heatflow_convergence_1}
  \end{subfigure}%
  \hfill%
  \begin{subfigure}[b]{.49\linewidth}
    \centering
  \tikzexternalenable%
  \tikzsetnextfilename{heatflow_convergence_2}%
  \begin{tikzpicture}[font = \plotfontsize]
  \pgfplotstableread{graphics/data/hf2d5_dt_convergence_coarseorth.dat}\tableSIM
  
  \begin{semilogyaxis}[%
    name   = states,
    width  = .75\textwidth,
    height = .115\textheight,
    scale only axis,
    xmin = 0,
    xmax = 50,
    ymin = 1e-9,
    ymax = 5e+1,
    xminorticks = false,
    yminorticks = false,
    xlabel = {iteration steps},
    ylabel = {},
    ylabel style   = {yshift = -.3em},
    scaled x ticks = false,
    x tick label style = {/pgf/number format/1000 sep={\,}},
    y tick label style = {/pgf/number format/1000 sep={\,}},
    cycle list name    = convergencelist
  ]
  
    \foreach \y in {1, 2, 3}{
      \addplot+ table[x index = 0, y index = \y] {\tableSIM};
    }
  \end{semilogyaxis}
\end{tikzpicture}%
  \tikzexternaldisable%

    \caption{large orthogonalization tolerance}
    \label{fig:heatflow_convergence_2}
  \end{subfigure}

  \vspace{0\baselineskip}
  \tikzexternalenable%
  \tikzsetnextfilename{convergence_legend}%
  \begin{tikzpicture}[font = \plotfontsize]
  \begin{axis}[%
    hide axis,
    width  = 1cm,
    height = 1cm,
    scale only axis,
    xmin = 0,
    xmax = 10,
    ymin = 0.5,
    ymax = 1.5,
    legend columns    = -1,
    legend cell align = {left},
    legend style      = {
      at     = {(0,0)},
      anchor = center,
      /tikz/every even column/.append style = {column sep = 0.5cm}}
  ]
    
    \pgfplotsset{cycle list name = convergencelist}
    \pgfplotsinvokeforeach{1, 2, 3}{\addplot coordinates {(0,0)};}
    \addlegendentry{$\lVert x(t_{j+1}) - V_{j}V_{j}^{\trans} x(t_{j+1}) \rVert_{2}$}
    \addlegendentry{$\lVert K_{j+1} - K_{j} \rVert_{2} /
      \lVert K_{j}\rVert_{2}$}
    \addlegendentry{$\lVert x(t_{j+1}) - \xs \rVert_{2}$}
  \end{axis}
\end{tikzpicture}%
  \tikzexternaldisable%

  \caption{ICI convergence behavior for discrete-time unstable heat flow:
    The convergence of ICI is illustrated in terms of the approximation error
    of new states in the current projection space $\Vcal_{j}$ and the
    relative change of the controller between steps.
    The peaks of the projection error indicate re-projections of the data,
    which are regularly needed for accurate orthogonalizations, but not
    if the data is more aggressively truncated with a large orthogonalization
    tolerance.
    In both cases, the distance of the trajectory to the steady state
    decreases, indicating stabilization.}
  \label{fig:heatflow_convergence}
\end{figure}
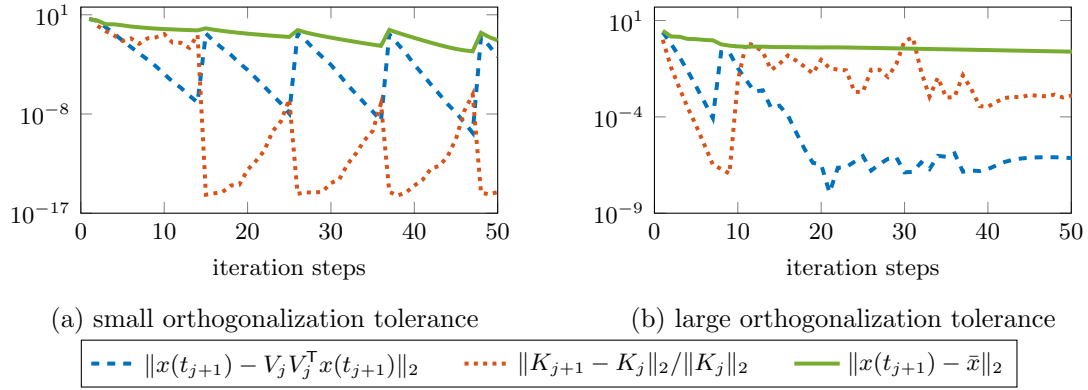

We then used the IMEX Euler scheme with the sampling time $\tau = 0.1$ to
discretize the linear system with the state dynamics implicitly in time and the
inputs explicitly in time.
The results for the three compared methods using similar numbers of model
queries are shown in \Cref{fig:heatflow_dt}, with ICI using $14$ queries,
ContInf from a single trajectory using $14$ model queries and
ContInf from four trajectories using $16$ queries.
Only ICI is able to provide a stabilizing controller for the system with
this small amount of evaluations, while the other two approaches fail.
By increasing the number of model evaluations, the other methods provide
stabilizing controllers after $22$ queries with a single trajectory and $24$
evaluations using four trajectories.

To investigate the general convergence behavior of the proposed ICI method,
we illustrate different quantities of interest appearing during the iterations
including the approximation error of the next state sample in the current
projection space, the relative change of the constructed controllers between
the iteration steps and the distance of the generated state trajectory to the
desired steady state in \Cref{fig:heatflow_convergence}.
The peaks in the projection error in \Cref{fig:heatflow_convergence_1} result
from re-projections applied to the data.
The results indicate that every time the projection error becomes comparably
small, the data becomes numerically challenging for the matrix inequality
solver leading to the re-projection method stepping in to improve the
conditioning of the data.
In \Cref{fig:heatflow_convergence_2}, a larger tolerance for the
orthogonalization method is used such that the constructed projection space
is increasing slower in terms of dimensions and the truncation of data is
smaller.
Therefore, the projection error has a higher minimum level and no
re-projections are needed to keep the data well-conditioned and informative.
Also, we observe that in both
\Cref{fig:heatflow_convergence_1,fig:heatflow_convergence_2} the distances
of the generated trajectories to the steady state of interest are
decreasing.
If no re-projections occur as in \Cref{fig:heatflow_convergence_2}, the
trajectories are monotonically decreasing.
This is not the case when re-projections occur as shown in
\Cref{fig:heatflow_convergence_1}.
However, in this case, the controller appears to be more efficient as the
distance to the steady state descreases faster.

%%%%%%%%%%%%%%%%%%%%%%%%%%%%%%%%%%%%%%%%%%%%%%%%%%%%%%%%%%%%%%%%%%%%%%%%%%%%%%%%

\subsection{Brazilian interconnected power system}

\begin{figure}[t]
  \centering
  \begin{subfigure}[b]{.49\linewidth}
    \raggedleft
  \tikzexternalenable%
  \tikzsetnextfilename{bips_dt_nofb}%
  \begin{tikzpicture}[font = \plotfontsize]
  \pgfplotstableread{graphics/data/bips_dt_sim_nofb.dat}\tableSIM
  
  \begin{axis}[%
    name   = states,
    width  = .73\textwidth,
    height = .1\textheight,
    scale only axis,
    xmin = 0,
    xmax = 500,
    ymin = -100,
    ymax = 800,
    xminorticks = false,
    yminorticks = false,
    xlabel = {time $\tau \cdot t$},
    ylabel = {outputs $y(\tau \cdot t)$},
    ylabel style   = {yshift = -.3em},
    scaled x ticks = false,
    x tick label style = {/pgf/number format/1000 sep={\,}},
    y tick label style = {/pgf/number format/1000 sep={\,}},
    cycle list name    = stateslist
  ]
  
    \foreach \y in {5, 7, 9}{
      \addplot+ table[x index = 0, y index = \y] {\tableSIM};
    }
  \end{axis}
\end{tikzpicture}%
  \tikzexternaldisable%

    \caption{uncontrolled simulation}
    \label{fig:bips_dt_nofb}
  \end{subfigure}%
  \hfill%
  \begin{subfigure}[b]{.49\linewidth}
    \raggedleft
  \tikzexternalenable%
  \tikzsetnextfilename{bips_dt_ici}%
  \begin{tikzpicture}[font = \plotfontsize]
  \pgfplotstableread{graphics/data/bips_dt_sim_ici.dat}\tableSIM
  
  \begin{axis}[%
    name   = states,
    width  = .73\textwidth,
    height = .1\textheight,
    scale only axis,
    xmin = 0,
    xmax = 500,
    ymin = -35,
    ymax = 60,
    xminorticks = false,
    yminorticks = false,
    xlabel = {time $\tau \cdot t$},
    ylabel = {outputs $y(\tau \cdot t)$},
    ylabel style   = {yshift = -.3em},
    scaled x ticks = false,
    x tick label style = {/pgf/number format/1000 sep={\,}},
    y tick label style = {/pgf/number format/1000 sep={\,}},
    cycle list name    = stateslist
  ]
  
    \foreach \y in {5, 7, 9}{
      \addplot+ table[x index = 0, y index = \y] {\tableSIM};
    }
  \end{axis}
\end{tikzpicture}%
  \tikzexternaldisable%

    \caption{ICI controller}
    \label{fig:bips_dt_ici}
  \end{subfigure}

  \vspace{0\baselineskip}
  \tikzexternalenable%
  \tikzsetnextfilename{bips_legend}%
  \begin{tikzpicture}[font = \plotfontsize]
  \begin{axis}[%
    hide axis,
    width  = 1cm,
    height = 1cm,
    scale only axis,
    xmin = 0,
    xmax = 10,
    ymin = 0.5,
    ymax = 1.5,
    legend columns    = -1,
    legend cell align = {left},
    legend style      = {
      at     = {(0,0)},
      anchor = center,
      /tikz/every even column/.append style = {column sep = 0.5cm}}
  ]
    
    \pgfplotsset{cycle list name = stateslist}
    \pgfplotsinvokeforeach{1, 2, 3}{\addplot coordinates {(0,0)};}
    \addlegendentry{output $y_{1}$}
    \addlegendentry{output $y_{2}$}
    \addlegendentry{output $y_{3}$}
  \end{axis}
\end{tikzpicture}%
  \tikzexternaldisable%

  \caption{Discrete-time power network: In the uncontrolled simulation, the
    output trajectories tend to infinity due to the instability of the
    power network.
    ICI provides a stabilizing controller after $20$ queries.
    The comparing low-dimensional controller inference setups did not provide
    a stabilizing controller within our limit of $300$ model queries.}
  \label{fig:bips_dt}
\end{figure}
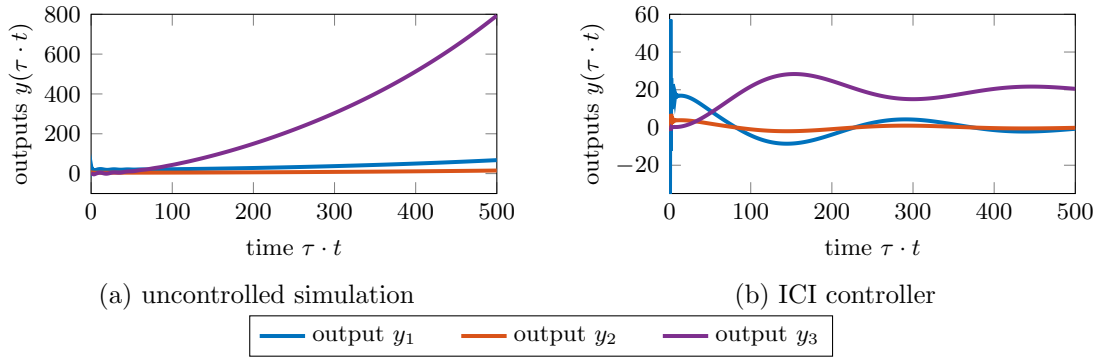

This example is a linearized model of the Brazilian interconnected power
network, which is unstable under heavy load
conditions~\cite{LosMRetal13, FreRM08}.
We consider the model data \textsf{bips98\_606}
from~\cite{morwiki_powersystems}, which yields a system of linear
differential-algebraic equations (DAEs) with $\nh = 7\,135$ and $\np = 4$.
To simulate this DAE system, we discretized the model using the IMEX Euler
scheme with the sampling time $\tau = 0.1$.

The simulation with the controller constructed from ICI with $20$
model queries is shown in \Cref{fig:bips_dt}.
ContInf based on a single or multiple trajectories using also $20$
model queries was not able to stabilize the system.
These simulation results are not shown here but can be found in the
accompanying code package~\cite{supWer25}.
In fact, we were not able to find any setup within our limit of $300$ model
queries in which the non-adaptive ContInf was able to provide a
stabilizing controller with either a single or multiple trajectories.

%%%%%%%%%%%%%%%%%%%%%%%%%%%%%%%%%%%%%%%%%%%%%%%%%%%%%%%%%%%%%%%%%%%%%%%%%%%%%%%%

\subsection{Tubular reactor}

\begin{figure}[t]
  \centering
  \begin{subfigure}[b]{.49\linewidth}
    \centering
  \tikzexternalenable%
  \tikzsetnextfilename{tubularreactor_ct_nofb}%
  \begin{tikzpicture}[font = \plotfontsize]  
  \begin{axis}[%
    name   = states,
    width  = .725\textwidth,
    height = .135\textheight,
    scale only axis,
    xmin = 0,
    xmax = 30,
    ymin = 0,
    ymax = 1,
    xminorticks = false,
    yminorticks = false,
    xlabel = {time $t$},
    ylabel = {spatial coordinates},
    ylabel style   = {yshift = -.3em},
    scaled x ticks = false,
    x tick label style = {/pgf/number format/1000 sep={\,}},
    y tick label style = {/pgf/number format/1000 sep={\,}}
  ]
  
    \addplot graphics[xmin = 0, xmax = 30, ymin = 0, ymax = 1]
        {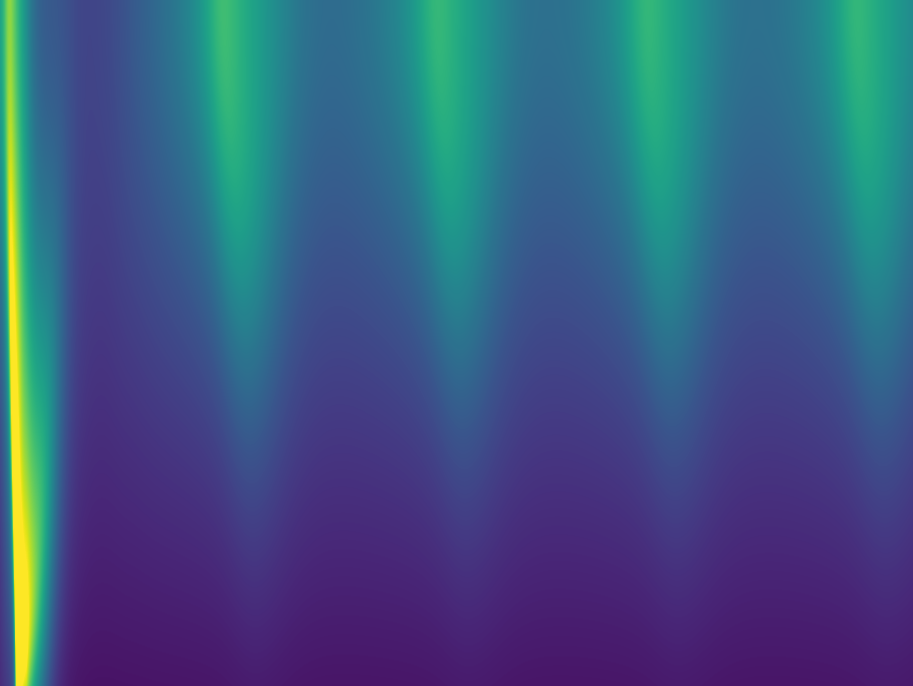};
  \end{axis}
\end{tikzpicture}%
  \tikzexternaldisable%

    \caption{CT: uncontrolled simulation}
    \label{fig:tubularreactor_ct_nofb}
  \end{subfigure}%
  \hfill%
  \begin{subfigure}[b]{.49\linewidth}
    \centering
  \tikzexternalenable%
  \tikzsetnextfilename{tubularreactor_ct_ici}%
  \begin{tikzpicture}[font = \plotfontsize]  
  \begin{axis}[%
    name   = states,
    width  = .725\textwidth,
    height = .135\textheight,
    scale only axis,
    xmin = 0,
    xmax = 30,
    ymin = 0,
    ymax = 1,
    xminorticks = false,
    yminorticks = false,
    xlabel = {time $t$},
    ylabel = {spatial coordinates},
    ylabel style   = {yshift = -.3em},
    scaled x ticks = false,
    x tick label style = {/pgf/number format/1000 sep={\,}},
    y tick label style = {/pgf/number format/1000 sep={\,}}
  ]
  
    \addplot graphics[xmin = 0, xmax = 30, ymin = 0, ymax = 1]
        {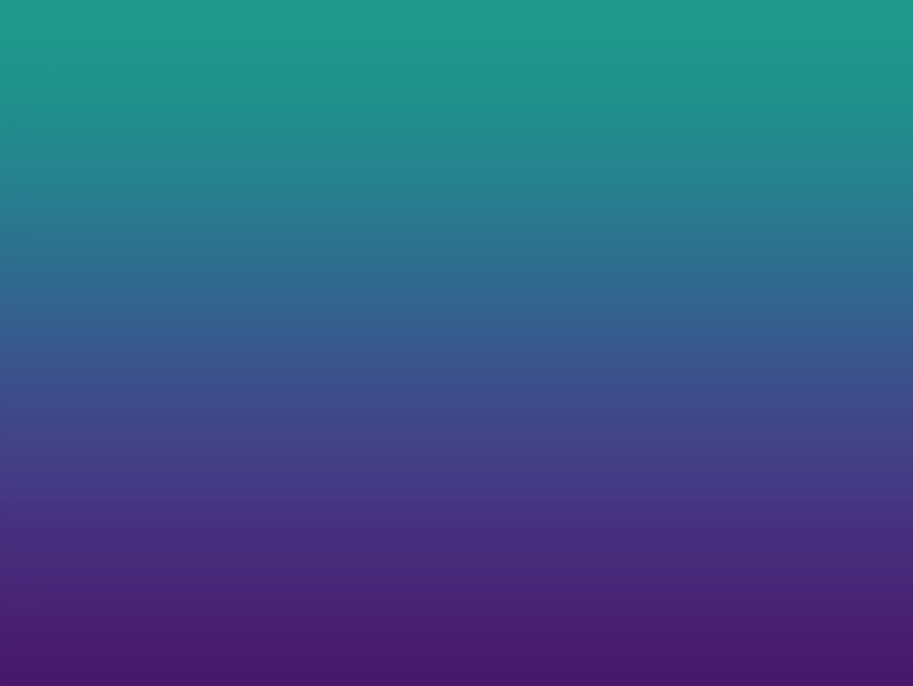};
  \end{axis}
\end{tikzpicture}%
  \tikzexternaldisable%

    \caption{CT: ICI controller}
    \label{fig:tubularreactor_ct_ici}
  \end{subfigure}
  \vspace{0\baselineskip}
  
  \begin{subfigure}[b]{.49\linewidth}
    \centering
  \tikzexternalenable%
  \tikzsetnextfilename{tubularreactor_dt_nofb}%
  \begin{tikzpicture}[font = \plotfontsize]  
  \begin{axis}[%
    name   = states,
    width  = .725\textwidth,
    height = .135\textheight,
    scale only axis,
    xmin = 0,
    xmax = 30,
    ymin = 0,
    ymax = 1,
    xminorticks = false,
    yminorticks = false,
    xlabel = {time $\tau \cdot t$},
    ylabel = {spatial coordinates},
    ylabel style   = {yshift = -.3em},
    scaled x ticks = false,
    x tick label style = {/pgf/number format/1000 sep={\,}},
    y tick label style = {/pgf/number format/1000 sep={\,}}
  ]
  
    \addplot graphics[xmin = 0, xmax = 30, ymin = 0, ymax = 1]
        {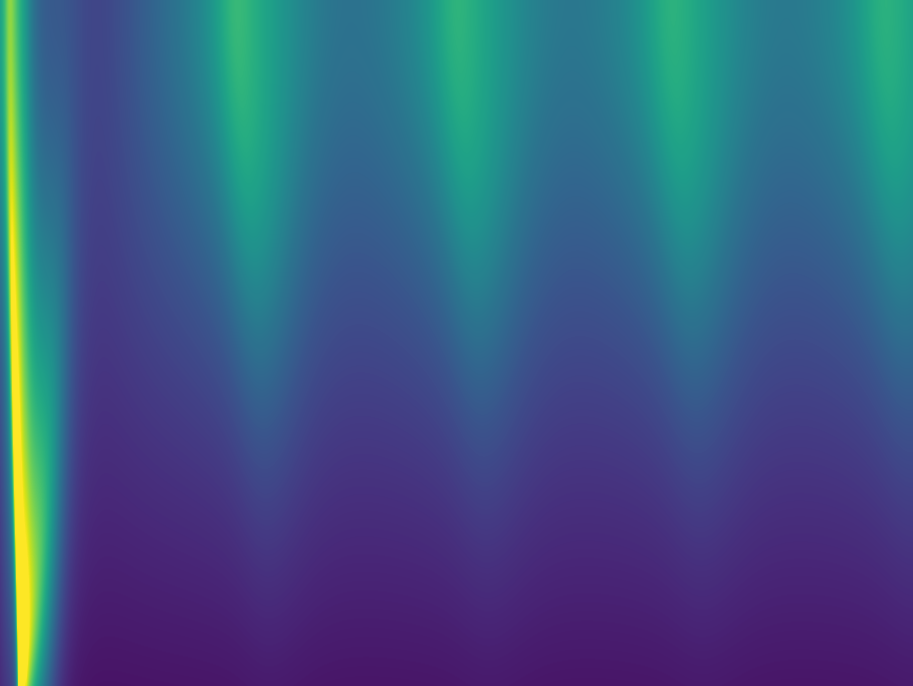};
  \end{axis}
\end{tikzpicture}%
  \tikzexternaldisable%

    \caption{DT: uncontrolled simulation}
    \label{fig:tubularreactor_dt_nofb}
  \end{subfigure}%
  \hfill%
  \begin{subfigure}[b]{.49\linewidth}
    \centering
  \tikzexternalenable%
  \tikzsetnextfilename{tubularreactor_dt_ici}%
  \begin{tikzpicture}[font = \plotfontsize]  
  \begin{axis}[%
    name   = states,
    width  = .725\textwidth,
    height = .135\textheight,
    scale only axis,
    xmin = 0,
    xmax = 30,
    ymin = 0,
    ymax = 1,
    xminorticks = false,
    yminorticks = false,
    xlabel = {time $\tau \cdot t$},
    ylabel = {spatial coordinates},
    ylabel style   = {yshift = -.3em},
    scaled x ticks = false,
    x tick label style = {/pgf/number format/1000 sep={\,}},
    y tick label style = {/pgf/number format/1000 sep={\,}}
  ]
  
    \addplot graphics[xmin = 0, xmax = 30, ymin = 0, ymax = 1]
        {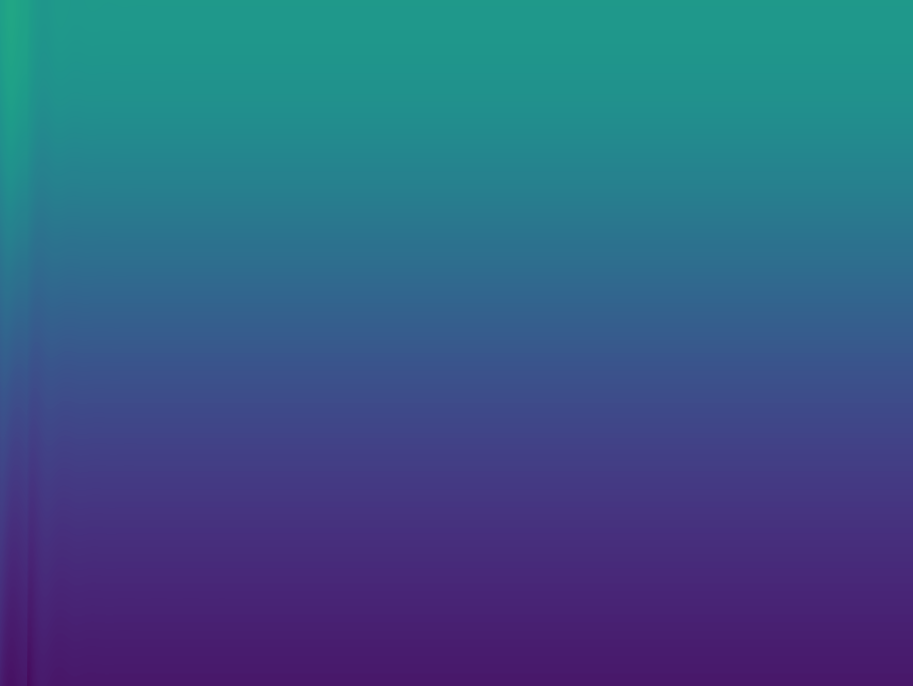};
  \end{axis}
\end{tikzpicture}%
  \tikzexternaldisable%

    \caption{DT: ICI controller}
    \label{fig:tubularreactor_dt_ici}
  \end{subfigure}

  \vspace{0\baselineskip}
  \tikzexternalenable%
  \tikzsetnextfilename{tubularreactor_legend}%
  \begin{tikzpicture}
  \node[draw = none, minimum width = 0cm, inner sep = 0cm](start){};
  \node(leg) at (start.north east) [anchor = north west]{\tikz
  \begin{axis}[%
    hide axis,
    scale only axis,
    width  = 10cm,
    height = .1cm,
    point meta min = 1.0,
    point meta max = 1.3,
    colorbar,
    colorbar horizontal,
    colorbar style = {
      at = {(.5, 0)},
      anchor = north},
    scaled x ticks = false,
    x tick label style = {/pgf/number format/fixed}]
  \end{axis};};
  \node[draw = none, minimum width = 0cm, inner sep = 0cm](end)
    at (leg.north east) [anchor = north west]{};
\end{tikzpicture}%
  \tikzexternaldisable%

  \caption{Continuous (CT) and discrete-time (DT) tubular reactor:
    Both in continuous as well as discrete-time, ICI constructs stabilizing
    controllers.
    In the discrete-time case, ICI performs better than the non-adaptive
    low-dimensional controller inference in terms of model
    queries by a factor of $1.5$.}
  \label{fig:tubularreactor}
\end{figure}

\begin{figure}[t]
  \centering
  \begin{subfigure}[b]{.49\linewidth}
    \centering
  \tikzexternalenable%
  \tikzsetnextfilename{tubularreactor_convergence_1}%
  \begin{tikzpicture}[font = \plotfontsize]
  \pgfplotstableread{graphics/data/tubularreactor_nl_dt_convergence.dat}\tableSIM
  
  \begin{semilogyaxis}[%
    name   = states,
    width  = .75\textwidth,
    height = .115\textheight,
    scale only axis,
    xmin = 0,
    xmax = 50,
    ymin = 1e-12,
    ymax = 1e+1,
    xminorticks = false,
    yminorticks = false,
    xlabel = {iteration steps},
    ylabel = {},
    ylabel style   = {yshift = -.3em},
    scaled x ticks = false,
    x tick label style = {/pgf/number format/1000 sep={\,}},
    y tick label style = {/pgf/number format/1000 sep={\,}},
    cycle list name    = convergencelist
  ]
  
    \foreach \y in {1, 2, 3}{
      \addplot+ table[x index = 0, y index = \y] {\tableSIM};
    }
  \end{semilogyaxis}
\end{tikzpicture}%
  \tikzexternaldisable%

    \caption{fine basis extension}
    \label{fig:tubularreactor_convergence_1}
  \end{subfigure}%
  \hfill%
  \begin{subfigure}[b]{.49\linewidth}
    \centering
  \tikzexternalenable%
  \tikzsetnextfilename{tubularreactor_convergence_2}%
  \begin{tikzpicture}[font = \plotfontsize]
  \pgfplotstableread{graphics/data/tubularreactor_nl_dt_convergence_coarseorth.dat}\tableSIM
  
  \begin{semilogyaxis}[%
    name   = states,
    width  = .75\textwidth,
    height = .115\textheight,
    scale only axis,
    xmin = 0,
    xmax = 50,
    ymin = 1e-10,
    ymax = 1e+1,
    xminorticks = false,
    yminorticks = false,
    xlabel = {iteration steps},
    ylabel = {},
    ylabel style   = {yshift = -.3em},
    scaled x ticks = false,
    x tick label style = {/pgf/number format/1000 sep={\,}},
    y tick label style = {/pgf/number format/1000 sep={\,}},
    cycle list name    = convergencelist
  ]
  
    \foreach \y in {1, 2, 3}{
      \addplot+ table[x index = 0, y index = \y] {\tableSIM};
    }
  \end{semilogyaxis}
\end{tikzpicture}%
  \tikzexternaldisable%

    \caption{coarse basis truncation}
    \label{fig:tubularreactor_convergence_2}
  \end{subfigure}

  \vspace{0\baselineskip}
  \tikzexternalenable%
  \tikzsetnextfilename{convergence_legend}%
  \begin{tikzpicture}[font = \plotfontsize]
  \begin{axis}[%
    hide axis,
    width  = 1cm,
    height = 1cm,
    scale only axis,
    xmin = 0,
    xmax = 10,
    ymin = 0.5,
    ymax = 1.5,
    legend columns    = -1,
    legend cell align = {left},
    legend style      = {
      at     = {(0,0)},
      anchor = center,
      /tikz/every even column/.append style = {column sep = 0.5cm}}
  ]
    
    \pgfplotsset{cycle list name = convergencelist}
    \pgfplotsinvokeforeach{1, 2, 3}{\addplot coordinates {(0,0)};}
    \addlegendentry{$\lVert x(t_{j+1}) - V_{j}V_{j}^{\trans} x(t_{j+1}) \rVert_{2}$}
    \addlegendentry{$\lVert K_{j+1} - K_{j} \rVert_{2} /
      \lVert K_{j}\rVert_{2}$}
    \addlegendentry{$\lVert x(t_{j+1}) - \xs \rVert_{2}$}
  \end{axis}
\end{tikzpicture}%
  \tikzexternaldisable%

  \caption{ICI convergence behavior for discrete-time tubular reactor:
    The convergence of ICI is illustrated in terms of the approximation error
    of new states in the current projection space $\Vcal_{j}$ and the
    relative change of the controller between steps.
    In both cases, the distance of the trajectory to the steady state
    decreases, indicating stabilization.}
  \label{fig:tubularreactor_convergence}
\end{figure}
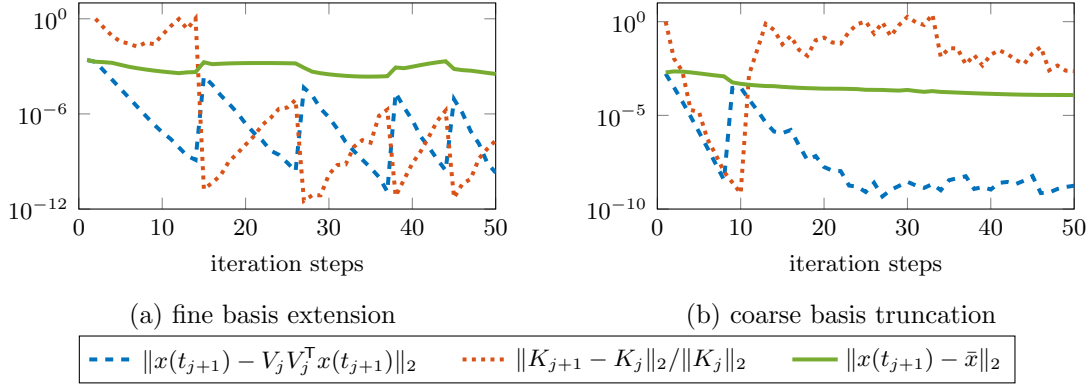

This example describes an exothermic reaction inside a tubular reactor with one
reactant~\cite{HeiP81, Zho12}.
Due to the influence of the temperature on the reaction speed, the system
enters an oscillating limit cycle behavior in the case of any disturbances.
Following the one-dimensional spatial discretization in~\cite{KraW19},
we have a system of $\nh = 3\,998$ nonlinear ODEs with $\np = 2$ inputs
that allow to control the reactant concentration at the beginning of the tube
and the temperature along the tube.
We consider the continuous-time as well as a discrete-time version of this
example.
The discrete-time model has been obtained using the IMEX Euler scheme with
sampling time $\tau = 0.01$, where the linear components are evaluated
implicitly in time and the nonlinear components and controls explicitly in time.

The continuous-time and discrete-time uncontrolled simulations and the results
for the new ICI approach are shown in \Cref{fig:tubularreactor}.
In the continuous-time case, for ICI we used the adaptive \texttt{ode15s}
solver to generate data at $6$ discrete time points.
The simulation of the resulting controller is shown in
\Cref{fig:tubularreactor_ct_ici}.
In the discrete-time case, ICI needed $8$ model queries to construct a
stabilizing controller.
The non-adaptive ContInf based on random inputs using the same
sample number of model evaluations cannot stabilize the system.
These results can be found in the accompanying code package~\cite{supWer25}.
However, increasing the number of allowed model queries to $12$ yields
stabilizing controllers via a single or two trajectories used in the
ContInf approach.
This is a $1.5\times$ increase in the number of model queries compared to ICI.

In \Cref{fig:tubularreactor_convergence}, we investigate the convergence
behavior of the ICI method in the discrete-time case beyond the minimum
amount of data needed to construct a stabilizing controller.
The plots show the approximation error of the states in the constructed
projection space, the relative change of the constructed controllers and the
distance of the state trajectory from the desired steady state.
Similar to the results shown in \Cref{sec:heatflow}, we observe a fast
convergence in terms of the projection error in the case of a small
orthogonalization tolerance and the regular applications of re-projections in
this case indicated by the spikes in the projection error in
\Cref{fig:tubularreactor_convergence_1}.
In the case of coarse basis extensions using a larger orthogonalization
tolerance shown in \Cref{fig:tubularreactor_convergence_2}, there is only one
re-projection needed in the beginning.
Also, the distance of the state trajectory to the steady state is decreasing
more consistently in \Cref{fig:tubularreactor_convergence_2} than in
\Cref{fig:tubularreactor_convergence_1}.
However, the decrease is not monotonic as it was for the heat flow example
(\Cref{fig:heatflow_convergence_2}) likely due to the oscillating nature of the
instabilities.

%%%%%%%%%%%%%%%%%%%%%%%%%%%%%%%%%%%%%%%%%%%%%%%%%%%%%%%%%%%%%%%%%%%%%%%%%%%%%%%%

\subsection{Crystal cluster growing}

\begin{figure}[t]
  \centering
  \begin{subfigure}[b]{.49\linewidth}
    \raggedleft
  \tikzexternalenable%
  \tikzsetnextfilename{todalattice_dt_nofb}%
  \begin{tikzpicture}[font = \plotfontsize]
  \pgfplotstableread{graphics/data/todalattice_nl_dt_sim_nofb.dat}\tableSIM
  
  \begin{axis}[%
    name   = states,
    width  = .73\textwidth,
    height = .1\textheight,
    scale only axis,
    xmin = 0,
    xmax = 100,
    ymin = -2,
    ymax = 2,
    restrict y to domain* = -2.0:2.0,
    xminorticks = false,
    yminorticks = false,
    xlabel = {time $\tau \cdot t$},
    ylabel = {outputs $y(\tau \cdot t)$},
    ylabel style   = {yshift = -.3em},
    scaled x ticks = false,
    x tick label style = {/pgf/number format/1000 sep={\,}},
    y tick label style = {/pgf/number format/1000 sep={\,}},
    cycle list name    = stateslist
  ]
  
    \foreach \y in {4, 5, 6}{
      \addplot+ table[x index = 0, y index = \y] {\tableSIM};
    }
  \end{axis}
\end{tikzpicture}%
  \tikzexternaldisable%

    \caption{uncontrolled simulation}
    \label{fig:todalattice_dt_nofb}
  \end{subfigure}%
  \hfill%
  \begin{subfigure}[b]{.49\linewidth}
    \raggedleft
  \tikzexternalenable%
  \tikzsetnextfilename{todalattice_dt_ici}%
  \begin{tikzpicture}[font = \plotfontsize]
  \pgfplotstableread{graphics/data/todalattice_nl_dt_sim_ici.dat}\tableSIM
  
  \begin{axis}[%
    name   = states,
    width  = .73\textwidth,
    height = .1\textheight,
    scale only axis,
    xmin = 0,
    xmax = 100,
    ymin = -2,
    ymax = 2,
    xminorticks = false,
    yminorticks = false,
    xlabel = {time $\tau \cdot t$},
    ylabel = {outputs $y(\tau \cdot t)$},
    ylabel style   = {yshift = -.3em},
    scaled x ticks = false,
    x tick label style = {/pgf/number format/1000 sep={\,}},
    y tick label style = {/pgf/number format/1000 sep={\,}},
    cycle list name    = stateslist
  ]
  
    \foreach \y in {4, 5, 6}{
      \addplot+ table[x index = 0, y index = \y] {\tableSIM};
    }
  \end{axis}
\end{tikzpicture}%
  \tikzexternaldisable%

    \caption{ICI controller}
    \label{fig:todalattice_dt_ici}
  \end{subfigure}
  \vspace{0\baselineskip}

  \begin{subfigure}[b]{.49\linewidth}
    \raggedleft
  \tikzexternalenable%
  \tikzsetnextfilename{todalattice_dt_sti}%
  \begin{tikzpicture}[font = \plotfontsize]
  \pgfplotstableread{graphics/data/todalattice_nl_dt_sim_sti.dat}\tableSIM
  
  \begin{axis}[%
    name   = states,
    width  = .73\textwidth,
    height = .1\textheight,
    scale only axis,
    xmin = 0,
    xmax = 100,
    ymin = -2,
    ymax = 2,
    restrict y to domain* = -2.0:2.0,
    xminorticks = false,
    yminorticks = false,
    xlabel = {time $\tau \cdot t$},
    ylabel = {outputs $y(\tau \cdot t)$},
    ylabel style   = {yshift = -.3em},
    scaled x ticks = false,
    x tick label style = {/pgf/number format/1000 sep={\,}},
    y tick label style = {/pgf/number format/1000 sep={\,}},
    cycle list name    = stateslist
  ]
  
    \foreach \y in {4, 5, 6}{
      \addplot+ table[x index = 0, y index = \y] {\tableSIM};
    }
  \end{axis}
\end{tikzpicture}%
  \tikzexternaldisable%

    \caption{ContInf (single trajectory)}
    \label{fig:todalattice_dt_sti}
  \end{subfigure}%
  \hfill%
  \begin{subfigure}[b]{.49\linewidth}
    \raggedleft
  \tikzexternalenable%
  \tikzsetnextfilename{todalattice_dt_mti}%
  \begin{tikzpicture}[font = \plotfontsize]
  \pgfplotstableread{graphics/data/todalattice_nl_dt_sim_mti.dat}\tableSIM
  
  \begin{axis}[%
    name   = states,
    width  = .73\textwidth,
    height = .1\textheight,
    scale only axis,
    xmin = 0,
    xmax = 100,
    ymin = -2,
    ymax = 2,
    restrict y to domain* = -2.0:2.0,
    xminorticks = false,
    yminorticks = false,
    xlabel = {time $\tau \cdot t$},
    ylabel = {outputs $y(\tau \cdot t)$},
    ylabel style   = {yshift = -.3em},
    scaled x ticks = false,
    x tick label style = {/pgf/number format/1000 sep={\,}},
    y tick label style = {/pgf/number format/1000 sep={\,}},
    cycle list name    = stateslist
  ]
  
    \foreach \y in {4, 5, 6}{
      \addplot+ table[x index = 0, y index = \y] {\tableSIM};
    }
  \end{axis}
\end{tikzpicture}%
  \tikzexternaldisable%

    \caption{ContInf (multiple trajectories)}
    \label{fig:todalattice_dt_mti}
  \end{subfigure}

  \vspace{0\baselineskip}
  \tikzexternalenable%
  \tikzsetnextfilename{todalattice_legend}%
  \begin{tikzpicture}[font = \plotfontsize]
  \begin{axis}[%
    hide axis,
    width  = 1cm,
    height = 1cm,
    scale only axis,
    xmin = 0,
    xmax = 10,
    ymin = 0.5,
    ymax = 1.5,
    legend columns    = -1,
    legend cell align = {left},
    legend style      = {
      at     = {(0,0)},
      anchor = center,
      /tikz/every even column/.append style = {column sep = 0.5cm}}
  ]
    
    \pgfplotsset{cycle list name = stateslist}
    \pgfplotsinvokeforeach{1, 2, 3}{\addplot coordinates {(0,0)};}
    \addlegendentry{output $y_{1}$}
    \addlegendentry{output $y_{2}$}
    \addlegendentry{output $y_{3}$}
  \end{axis}
\end{tikzpicture}%
  \tikzexternaldisable%

  \caption{Discrete-time crystal cluster:
    The three clusters quickly break apart and diverge from each other in the
    uncontrolled simulation and only the ICI approach yields a suitable
    controller using $16$ model evaluations.
    About $4.4\times$ more queries than in ICI are needed by ContInf
    using unguided single trajectory data.}
  \label{fig:todalattice_dt}
\end{figure}
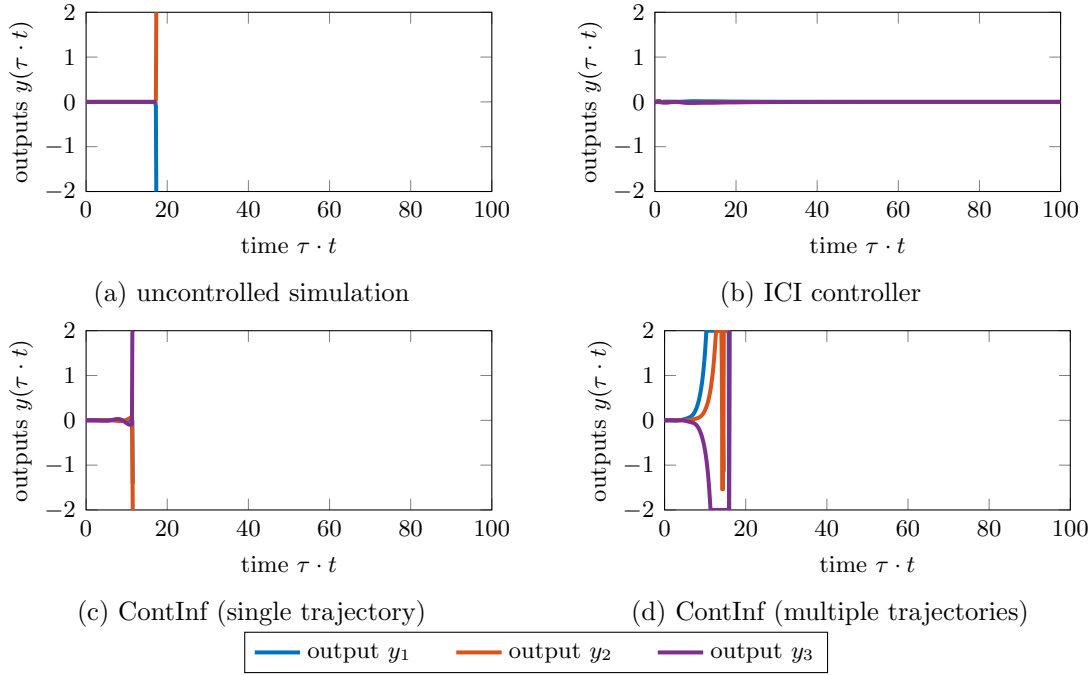

The Toda lattice model is used to describe one-dimensional
crystals~\cite{Tod67, Wer21} using classical physics laws from solid-state
mechanics about mass-spring-damper systems.
Here, we consider the case of a crystalization process in which three major
crystal clusters have been formed already.
Stabilization is needed to hold these clusters in place to allow for the
full crystal to form; see~\cite{WerP25} for further details of the setup
including the describing equations.
In our experiments, we are using $2\,000$ particles leading to a system of
$\nh = 4\,000$ nonlinear ODEs with $\np = 3$ inputs to change the displacement
of the three clusters.
For the discretization of the model, we use the IMEX Euler scheme with
sampling time $\tau = 0.1$, where the linear parts are handled implicitly and
the nonlinear parts and controls explicitly.

The simulations of the model can be seen in \Cref{fig:todalattice_dt}.
The outputs shown are the mean velocities of the particles of the clusters
and any blow-up behavior implies that the clusters are drifting apart
indefinitely.
ICI provides a stabilizing controller after $16$ model queries.
The non-adaptive ContInf using the same number of
queries with a single or multiple trajectories from random inputs are not
stabilizing as shown in
\Cref{fig:todalattice_dt_sti,fig:todalattice_dt_mti}.
Using a single trajectory only, $70$ model evaluations are needed to
construct a stabilizing controller, which is around $4.4\times$ more queries
than ICI needs.
Allowing data from three trajectories reduces the number of model
queries for the design of a stabilizing controller to $30$, which is still
around $1.8\times$ more than needed for ICI.

%%%%%%%%%%%%%%%%%%%%%%%%%%%%%%%%%%%%%%%%%%%%%%%%%%%%%%%%%%%%%%%%%%%%%%%%%%%%%%%%

\subsection{Laminar flow behind an obstacle}

\begin{figure}[t]
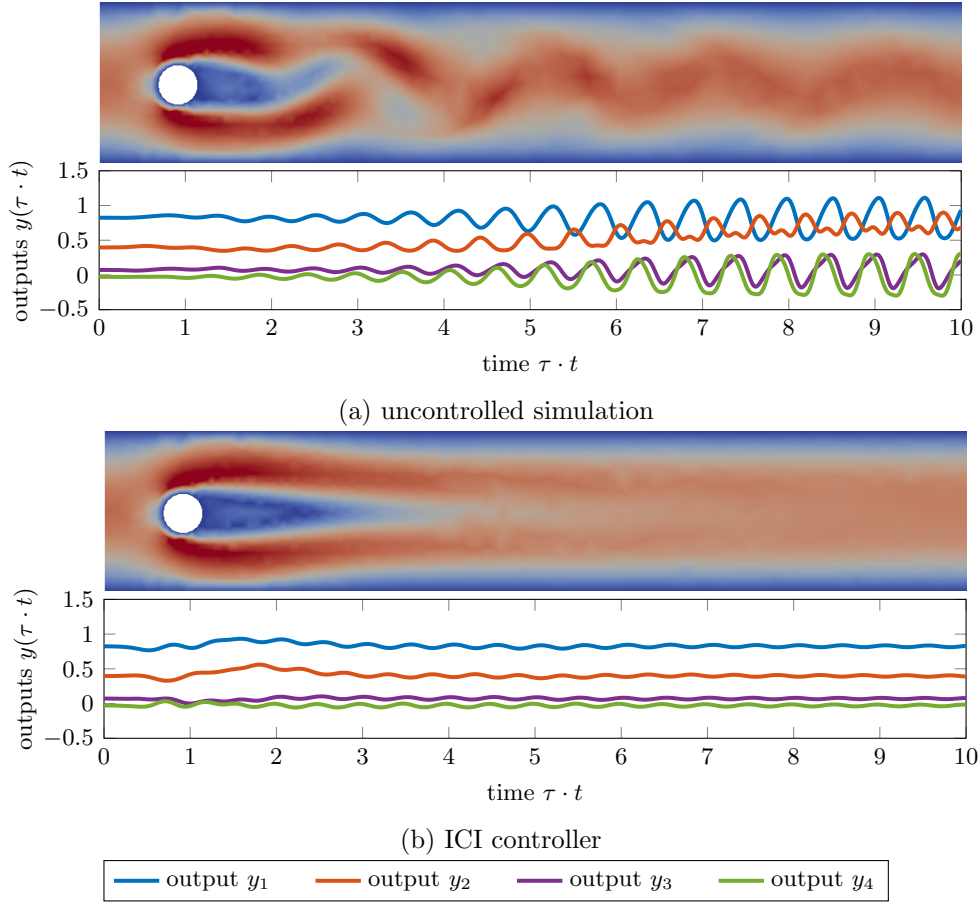

  \centering
  \begin{subfigure}[b]{.95\linewidth}
    \centering
  \tikzexternalenable%
  \tikzsetnextfilename{cylinderwake_dt_nofb}%
  \input{graphics/cylinderwake_dt_nofb.tikz}%
  \tikzexternaldisable%

    \caption{uncontrolled simulation}
    \label{fig:cylinderwake_dt_nofb}
  \end{subfigure}
  \vspace{0\baselineskip}

  \begin{subfigure}[b]{.95\linewidth}
    \centering
  \tikzexternalenable%
  \tikzsetnextfilename{cylinderwake_dt_ici}%
  \input{graphics/cylinderwake_dt_ici.tikz}%
  \tikzexternaldisable%

    \caption{ICI controller}
    \label{fig:cylinderwake_dt_ici}
  \end{subfigure}

  \vspace{0\baselineskip}
  \tikzexternalenable%
  \tikzsetnextfilename{cylinderwake_legend}%
  \begin{tikzpicture}[font = \plotfontsize]
  \begin{axis}[%
    hide axis,
    width  = .75\textwidth,
    height = .1\textheight,
    scale only axis,
    xmin = 0,
    xmax = 10,
    ymin = 0.5,
    ymax = 1.5,
    legend columns = 5, 
    legend style = {
      at     = {(0,0)},
      anchor = center,
      /tikz/every even column/.append style = {column sep = 0.5cm}}
  ]
    
    \pgfplotsset{cycle list name = stateslist}
    \pgfplotsinvokeforeach{1, 2, 3, 4}{\addplot coordinates {(0,0)};}
    \addlegendentry{output $y_{1}$}
    \addlegendentry{output $y_{2}$}
    \addlegendentry{output $y_{3}$}
    \addlegendentry{output $y_{4}$}
  \end{axis}
\end{tikzpicture}%
  \tikzexternaldisable%

  \caption{Laminar flow behind an obstacle:
    The upper plots show the velocity field in the last time step and the
    lower plots the mean flow velocities at the end of the channel.
    ICI can stabilizes the system using only $10$ system queries, while for
    non-adaptive ContInf $25 \times$ or $12 \times$ more model
    queries are needed to construct stabilizing controllers, depending on the
    use of only a single or six state trajectories.}
  \label{fig:cylinderwake_dt}
\end{figure}

As final example, we consider the laminar flow behind an obstacle, which
is described by the incompressible Navier-Stokes equations~\cite{BehBH17}.
At medium Reynolds numbers---here $\Reyn = 90$---the non-zero steady state
tends to unstable oscillatory vortex shedding.
We consider the example matrices given in~\cite{BehBH17}, providing a spatially
discretized model with $\nh = 6\,618$ nonlinear DAEs and $\np = 6$ inputs that
allow the change of the flow velocities in a small rectangle behind the
obstacle.
For the time simulation, the system is discretized in time via the IMEX Euler
scheme with sampling time $\tau = 0.0025$ and where the linear parts are handled
implicitly and the nonlinear components and controls explicitly.

The uncontrolled simulation and the results for the ICI-based feedback
controller are shown in \Cref{fig:cylinderwake_dt} with snapshots of the
velocity fields from the last time step and trajectories of the
mean velocities at the end of the channel.
ICI needed only $10$ model queries to provide a stabilizing controller for the
system.
In contrast, to stabilize the system with the non-adaptive
ContInf approach, either $250$ model queries in a single state trajectory or
$120$ queries for multiple trajectories are needed to construct stabilizing
controllers.
These numbers of model queries are $25$ and $12$ times larger than for ICI,
respectively.

%%%%%%%%%%%%%%%%%%%%%%%%%%%%%%%%%%%%%%%%%%%%%%%%%%%%%%%%%%%%%%%%%%%%%%%%%%%%%%%%
% CONCLUSIONS.                                                                 %
%%%%%%%%%%%%%%%%%%%%%%%%%%%%%%%%%%%%%%%%%%%%%%%%%%%%%%%%%%%%%%%%%%%%%%%%%%%%%%%%

\section{Conclusions}%
\label{sec:conclusions}

We introduced an adaptive sampling scheme for the construction of informative
data sets for the task of constructing stabilizing feedback controllers.
The approach is inspired by sufficient conditions for data informativity using
stabilizing control signals as system inputs.
Thereby, a sequence of stabilizing controllers is constructed over nested
low-dimensional subspaces based on state data samples that are obtained by
exciting the system using previous stabilizing controllers.
The numerical experiments demonstrate that up to one order
of magnitude fewer samples are needed by the proposed approach
compared to unguided data collection.
The method as presented in this work solely relies on forward evaluations of
the underlying model.
Future work includes combining the proposed adaptive sampling approach with
adjoint information, which can further reduce the required number of data
samples~\cite{WerP23a}.
This is of particular interest in applications where the instabilities affect
high-dimensional subspaces such as in chaotic systems.

%%%%%%%%%%%%%%%%%%%%%%%%%%%%%%%%%%%%%%%%%%%%%%%%%%%%%%%%%%%%%%%%%%%%%%%%%%%%%%%%
% *** ACKNOWLEDGEMENTS ***                                                     %
%%%%%%%%%%%%%%%%%%%%%%%%%%%%%%%%%%%%%%%%%%%%%%%%%%%%%%%%%%%%%%%%%%%%%%%%%%%%%%%%

\section*{Acknowledgments}%
\addcontentsline{toc}{section}{Acknowledgments}

Parts of this work were carried out while Werner was at the Courant Institute
of Mathematical Sciences, New York University, USA.

Peherstorfer was partially supported by the Air Force Office of Scientific
Research (AFOSR), USA, award FA9550-21-1-0222 and FA9550-24-1-0327
(Dr.\ Fariba Fahroo).

%%%%%%%%%%%%%%%%%%%%%%%%%%%%%%%%%%%%%%%%%%%%%%%%%%%%%%%%%%%%%%%%%%%%%%%%%%%%%%%%
% REFERENCES.                                                                  %
%%%%%%%%%%%%%%%%%%%%%%%%%%%%%%%%%%%%%%%%%%%%%%%%%%%%%%%%%%%%%%%%%%%%%%%%%%%%%%%%

\addcontentsline{toc}{section}{References}
\bibliographystyle{plainurl}
\bibliography{bibtex/myref}

\end{document}